\documentclass[12pt,a4paper]{article}
\usepackage{amssymb}
\usepackage{amsmath}
\usepackage{amsthm}
\usepackage[T1]{fontenc}
\usepackage[utf8]{inputenc}
\usepackage[all]{xy}
\usepackage{mathrsfs}
\usepackage[usenames,dvipsnames]{color}
\usepackage{enumitem}
\usepackage{cancel}
\usepackage{hyperref}
\usepackage{memhfixc}
\usepackage{mathtools}
\usepackage{wasysym} 

\newcommand{\checkedproof}{}

\DeclareMathAlphabet\mathbb{U}{msb}{m}{n}
\usepackage{eucal}
\newtheorem{theorem}{Theorem}[section]
\newtheorem{corollary}[theorem]{Corollary}
\newtheorem{proposition}[theorem]{Proposition}
\newtheorem{lemma}[theorem]{Lemma}

\theoremstyle{definition}
\newtheorem{definition}[theorem]{Definition}
\theoremstyle{remark}
\newtheorem{remark}[theorem]{Remark}

\newcommand{\macrocolor}{}
\def\notationcolor{}
\newcommand{\proofs}[1]{}

\newcommand{\Gq}[2]{{\macrocolor\cG^{#2}(#1)}}
\newcommand{\Gb}[2]{{\macrocolor\cE^{#2}(#1)}}
\newcommand{\Gm}[2]{{\macrocolor\cE^{#2}_{\cM}(#1)}}
\newcommand{\Gn}[2]{{\macrocolor\cN^{#2}(#1)}}
\newcommand{\Gbloc}[2]{{\macrocolor\cE^{#2}_{loc}(#1)}}
\newcommand{\Gmloc}[2]{{\macrocolor\cE^{#2}_{\cM,loc}(#1)}}
\newcommand{\Gnloc}[2]{{\macrocolor\cN^{#2}_{loc}(#1)}}
\newcommand{\Gloc}[2]{{\macrocolor\cG^{#2}_{loc}(#1)}}
\newcommand{\SGloc}[1]{{\macrocolor\cG^{#1}_{loc}}}
\newcommand{\Gtb}[2]{{\macrocolor\cT_{\cE}^{#2}(#1)}}
\newcommand{\Gt}[2]{{\macrocolor\cT_{\cG}^{#2}(#1)}}

\newcommand{\Gd}{{\macrocolor\cG^d(\Omega)}} 
\newcommand{\sGd}{{\macrocolor\cG^d}}

\newcommand{\shgrs}{{\macrocolor\widehat\cG^r_s}} 
\newcommand{\hgrs}[1]{{\macrocolor\shgrs(#1)}} 
\newcommand{\shg}{{\macrocolor\widehat\cG}} 
\newcommand{\hg}[1]{{\macrocolor{\shg(#1)}}} 

\newcommand{\Gnb}[1]{{\macrocolor\cE(#1)}}
\newcommand{\Gnm}[1]{{\macrocolor\cE_M(#1)}}
\newcommand{\Gnn}[1]{{\macrocolor\cN(#1)}}
\newcommand{\Gen}[1]{{\macrocolor\cG(#1)}} 

\newcommand{\cTrsM}{{\macrocolor\mathcal{T}^r_s(M)}}

\newcommand{\Linc}{{\macrocolor\mathcal{L}}} 
\newcommand{\Lincb}{{\macrocolor\mathcal{L}}} 
\newcommand{\Lincbmult}[1]{{\macrocolor\mathcal{L}^{#1}}} 
\newcommand{\Linbb}{{\macrocolor\mathrm{L}}} 
\newcommand{\Linbbmult}[1]{{\macrocolor\mathrm{L}^{#1}}} 

\newcommand{\Hom}{{\macrocolor\operatorname{Hom}}}

\newcommand{\csn}[1]{{\macrocolor{\operatorname{csn}}}(#1)} 
\newcommand{\ud}{\mathrm{d}} 
\newcommand{\fX}{\mathfrak{X}} 
\newcommand{\cD}{\mathcal{D}}
\newcommand{\Lsk}{\Lie^{\mathrm{SK}}}

\newcommand{\TO}[1]{{\macrocolor S(#1)}} 
\newcommand{\TOz}[1]{{\macrocolor S^0(#1)}} 
\newcommand{\TOv}[1]{{\macrocolor S(#1)}} 
\newcommand{\TOvl}[1]{{\macrocolor S^1(#1)}} 
\newcommand{\TOvn}[1]{{\macrocolor N(#1)}} 
\newcommand{\TOvz}[1]{{\macrocolor S^0(#1)}} 
\newcommand{\STOvl}{{\macrocolor S^1}} 
\newcommand{\STOvz}{{\macrocolor S^0}} 
\newcommand{\TOqv}[1]{{\macrocolor \widetilde S(#1)}} 
\newcommand{\TOqvl}[1]{{\macrocolor \widetilde S^1(#1)}} 
\newcommand{\TOqvz}[1]{{\macrocolor \widetilde S^0(#1)}} 

\newcommand{\pd}{\partial}
\newcommand{\e}{\varepsilon}

\newcommand{\bC}{\mathbb{C}}
\newcommand{\bE}{\mathbb{E}}
\newcommand{\bF}{\mathbb{F}}

\newcommand{\bK}{\mathbb{K}}
\newcommand{\bN}{\mathbb{N}}
\newcommand{\bR}{\mathbb{R}}

\newcommand{\Riem}{\mathrm{\mathbf{R}}}

\newcommand{\cE}{\mathcal{E}}
\newcommand{\cF}{\mathcal{F}}
\newcommand{\cG}{\mathcal{G}}

\newcommand{\cM}{\mathcal{M}}
\newcommand{\cN}{\mathcal{N}}

\newcommand{\cT}{\mathcal{T}}

\newcommand{\csub}{\subset \subset}
\newcommand{\coleq}{\mathrel{\mathop:}=}
\newcommand{\eqcol}{=\mathrel{\mathop:}}
\DeclareMathOperator{\Fl}{Fl}

\DeclareMathOperator{\ev}{ev}

\DeclareMathOperator{\id}{id}
\DeclareMathOperator{\Vol}{Vol}

\DeclareMathOperator{\supp}{supp}

\DeclareMathOperator{\pr}{pr}

\newcommand{\Lie}{\mathrm{L}}

\providecommand{\norm}[1]{\left\lVert#1\right\rVert}
\providecommand{\abso}[1]{\left\lvert#1\right\rvert}

  \makeatletter
    \let\OLD@otimes\otimes
    \def\opotimes{\@ifnextchar_\otimesWSB\otimesNSB}
    \def\otimesWSB_#1{\@ifnextchar^{\otimesWSBWXP_{#1}}{\otimesWSBNXP_{#1}}}
    \def\otimesWSBNXP_#1{\mathbin{\operatorname*{\OLD@otimes}_{#1}}}
    \def\otimesWSBWXP_#1^#2{\mathbin{\operatorname*{\OLD@otimes}_{#1}^{#2}}}
    \def\otimesNSB{\@ifnextchar^\otimesWXP\OLD@otimes}
    \def\otimesWXP^#1{\@ifnextchar_{\otimesWXPWSB^{#1}}{\otimesWXPNSB^{#1}}}
    \def\otimesWXPNSB^#1{\mathbin{\operatorname*{\OLD@otimes}^{#1}}}
    \def\otimesWXPWSB^#1_#2{\mathbin{\operatorname*{\OLD@otimes}^{#1}_{#2}}}
  \makeatother

\newcommand{\cH}{{\mathscr{H}}}
\newcommand{\rSO}{\rho^{SO}} 

\newcommand{\csna}{\mathfrak{p}}
\newcommand{\csnb}{\mathfrak{q}}
\newcommand{\csnc}{\mathfrak{r}}

\newcommand{\SK}[1]{{\macrocolor\mathrm{SK}(#1)}}
\newcommand{\Lso}{\Lie^{\mathrm{SO}}}
\newcommand{\Lieh}{{\widehat\Lie}}
\newcommand{\Liet}{{\widetilde\Lie}}
\newcommand{\SO}[1]{{\macrocolor\mathrm{SO}(#1)}}

\newcommand{\VSO}[1]{{\macrocolor\mathrm{VSO}(#1)}}

\DeclareMathOperator{\carr}{carr}
\newcommand{\D}{\mathrm{D}}

\def\mytitle{Nonlinear generalized sections\\of vector bundles}

\begin{document}
\author{E.~A.~Nigsch}
\date{}
\title{\mytitle}

\maketitle


\begin{abstract}
 We present an extension of J.-F.\ Colombeau's theory of nonlinear generalized functions to spaces of generalized sections of vector bundles. Our construction builds on classical functional analytic notions, which is the key to having a canonical geometric embedding of vector bundle valued distributions into spaces of generalized sections. This permits to have tensor products, invariance under diffeomorphisms, covariant derivatives and the sheaf property. While retaining as much compatibility to L.~Schwartz' theory of distributions as possible, our theory provides the basis for a rigorous and general treatment of singular pseudo-Riemannian geometry in the setting of Colombeau nonlinear generalized functions.
\end{abstract}


{\bfseries MSC2100 Classification: } 46T30




\section{Introduction}

The theory of distributions, founded by S.~Sobolev \cite{0014.05902} and L.~Schwartz \cite{TD}, has been very successfully applied in many fields such as the study of partial differential equations, Fourier analysis, engineering and theoretical physics (see, e.g., \cite{zbMATH02208228,Hoermander,zbMATH03187905,Treves,zbMATH03421400}). However, it does not adapt well to nonlinear problems because of the inherent difficulty to define nonlinear operations on distributions.

The problem of \emph{multiplication of distributions} has been approached in various ways which mainly fall into two classes: either one defines the product only for certain pairs of distributions, or one embeds the space $\cD'$ of distributions into an algebra. The core issue to overcome in this is a certain algebraic incompatibility between differentiation, multiplication and singular functions that is made precise by L.~Schwartz' \emph{impossibility result} \cite{Schwartz}, which states that there cannot be an associative commutative differential algebra containing $\cD'$ as a linear subspace such that the constant function $1$ becomes the multiplicative unit and the partial derivatives of distributions as well as the pointwise product of continuous functions are preserved in the algebra.

While this result was commonly interpreted to preclude any reasonable multiplication of distributions, one can in fact construct algebras of generalized functions containing $\cD'$ by weakening the above requirements in one form or another; see \cite{Rosinger,MOBook,GKOS} for a comprehensive overview on what is possible. A particularly well-known and widely used approach has been introduced by J.~F.~Colombeau (\cite{ColNew, ColElem}), who constructed differential algebras of generalized functions containing $\cD'$ and preserving the product not of continuous but of \emph{smooth} (i.e., infinitely differentiable) functions.

These \emph{Colombeau algebras}, as they are commonly called, have been developed further (\cite{Biagioni,colmult, GKOS, zbMATH01226424, MOBook}) and applied successfully in a wide variety of fields, ranging from linear and nonlinear partial differential equations with singular data or singular coefficients (see \cite{Nedeljkov} for a recent survey) over pseudodifferential operators and Fourier integral operators with non-smooth symbols (\cite{zbMATH05208728,zbMATH05566552,zbMATH05357102,zbMATH05002685}) to the investigation of topological and algebraic structures in Colombeau generalized function spaces (\cite{claudia1,claudia2,zbMATH05356718,zbMATH05657323,zbMATH05768499}).
A particular development in the theory of Colombeau algebras concerns a geometric formulation of the theory with the aim of giving a comprehensive framework for problems of non-smooth differential geometry with applications in geophysics, Lie group analysis of differential equations or general relativity (\cite{found,global,MKnonsmooth,KSVintrins,Kgenmf,GKOS, GroupAnalysis, genrel}).

In this geometric setting there are several highly interesting, physically relevant results which could not be obtained by distributional methods alone (\cite{zbMATH01019155, geodesics,Steinbauer:98,JVClarke,zbMATH01364124,genconcurv,Kgenmf,KSVintrins,zbMATH02073427}); however, progress in the geometric theory of Colombeau algebras, in particular the study of generalized sections of vector bundles and generalized pseudo-Riemannian geometry, has been mostly limited to the \emph{special} variant of these generalized function spaces so far (\cite{foundgeom,genpseudo}). This is a simplified variant of the theory which is easier to calculate with, but has several drawbacks which in a sense preclude genuine geometrical results (cf.~\cite[Section 3.2.2]{GKOS}); in particular, for special Colombeau algebras there is no induced action of diffeomorphisms extending the classical pullback of distributions, there is no canonical embedding of distributions, and no embedding of distributions can commute with arbitrary Lie derivatives.

In contrast to the special variant one has the so-called \emph{full} variant of Colombeau algebras. There, the drawbacks of the special algebra just mentioned do not appear but are traded in for a more complicated technical apparatus needed for the formulation of the theory. A diffeomorphism invariant local theory, substantially based on previous work of Colombeau and Meril \cite{colmani} and Jel\'inek \cite{Jelinek}, was for the first time obtained in \cite{found}, where the full diffeomorphism invariant algebra $\Gd$ on open subsets $\Omega \subseteq \bR^n$ was presented. The generalization to manifolds, which involved a change of formalism because $\Gd$ still was based on the linear structure of $\bR^n$, was accomplished in \cite{global} with the introduction of the full Colombeau algebra $\hg M$ on any manifold $M$.

In a next step it was naturally very desirable to have an extension to a theory of generalized sections of vector bundles, and in particular a theory of nonlinear generalized tensor fields suitable for applications in (pseudo-)Riemannian geometry. The basic problem encountered in this case is that one cannot use a coordinatewise embedding: simply defining $\hgrs M \coleq \hg M \otimes_{C^\infty(M)} \cTrsM$ for generalized $(r,s)$-tensor fields (where $\cTrsM$ is the space of smooth $(r,s)$-tensor fields on $M$) cannot succeed due to a consequence of the Schwartz impossibility result (\cite[Proposition 4.1]{global2}). The underlying reason, which will be detailed in Section \ref{subsec_smoothop}, is that the embedding of distributions into Colombeau algebras always involves some kind of regularization, but in order to regularize non-smooth or distributional sections of a vector bundle one needs to transport vectors between different points of the manifold (cf.~\cite{papernew}); based on these ideas, in \cite{global2} an algebra $\bigoplus_{r,s} \shgrs$ of generalized tensor fields incorporating the necessary modifications was constructed.

This full generalized tensor algebra, however, also suffered from serious drawbacks quite different from those of special algebras:
\begin{enumerate}[label=(\roman*)]
 \item $\shgrs$ inherits all the technical difficulties from $\sGd$ and $\shg$ and adds even more on top of them, which makes it rather inaccessible for non-specialists and precludes easy applications.
 \item $\shgrs$ is not a sheaf: the corresponding proof which worked in all previous algebras (cf., e.g., \cite[Section 8]{found}) breaks down due to the failure of the test objects to be `localizing' in a certain sense.
 \item There is no way to define a meaningful covariant derivative $\nabla_X$ on $\hgrs M$ that is $C^\infty(M)$-linear in the vector field $X$, which would be an indispensable necessity for geometrical applications like the definition of generalized curvature (cf.~\cite{papernew}).
\end{enumerate}
Summing up, despite its achievements $\shgrs$ was still unsatisfactory and raised fundamental questions about how to proceed, in general.

The latest turning point in this development was the introduction of a functional analytic approach to Colombeau algebras (developed for the scalar case in \cite{papernew}) that both unifies and simplifies previous constructions on a conceptual level as well as provides the structural framework for developing the global vector-valued case in a natural way, overcoming points (i)--(iii).

The aim of this article is to give the details of this construction for the case of generalized sections of vector bundles on manifolds. We will, in a very general way, introduce these spaces and show how fundamental concepts like tensor products, the sheaf property, covariant derivatives, pseudo-Riemannian metrics and the curvature tensor can be obtained in a natural and efficient way in this context. This constitutes a new foundation of the geometric theory of Colombeau algebras and its application to singular (pseudo-)Riemannian geometry.

\section{Preliminaries}

In this article, the following conventions will be used. $\bN$, $\bN_0$, $\bR$ and $\bC$ are the sets of positive and nonnegative integers, real and complex numbers, respectively. On any $\bR^n$, $\norm{.}$ denotes the Euclidean norm. We set $\bK \coleq \bR$ or $\bK \coleq \bC$ throughout, depending on whether we want to consider real or complex distributions, and $I \coleq (0,1]$.
A family indexed by a set $J$ is denoted by $(x_j)_{j \in J}$ or simply $(x_j)_j$ if the index set is clear from the context. The equivalence class of an element $x$ with respect to some given equivalence relation is denoted by $[x]$, while $[y \mapsto f(y)]$ denotes a function $f$ depending on a variable $y$. We use the usual multiindex notation $\pd^\alpha = \pd_1^{\alpha_1} \dotsm \pd_n^{\alpha_n}$ for $\alpha = (\alpha_1, \dotsc, \alpha_n) \in \bN_0^n$ and $n \in \bN$, where $\pd_i$ is the $i$-th partial derivative, and set $\abso{\alpha} = \alpha_1 + \dotsc + \alpha_n$. If the differentiation variable needs to be specified we write $\pd_x^\alpha$ or $\pd_{x_i}$. We employ the Landau notation $f_\e = O(g_\e)$ ($\e \to 0$) for nets $(f_\e)_\e$ and $(g_\e)_\e$. $\id$ denotes the identity map. The restriction of a function $f$ to a set $U$ is denoted by $f|_U$. $\carr f$ and $\supp f$ denote the set $\{ x: f(x) \ne 0 \}$ and its closure, respectively. We write $U \csub V$ if the closure $\overline{U}$ is compact and contained in the interior of $V$.

As basic references for locally convex spaces we use \cite{Jarchow,zbMATH03230708}. Let $\bE, \bF, \dotsc$ be locally convex spaces over $\bK$; by $\Lincb(\bE,\bF)$ and $\Linbb(\bE,\bF)$ we denote the space of linear maps from $\bE$ to $\bF$ which are continuous or bounded, respectively. Similarly, $\Lincbmult k(\bE_1 \times \dotsc \times \bE_k,\bF)$ and $\Linbbmult k (\bE_1 \times \dotsc \times \bE_k,\bF)$ denote the space of continuous or bounded $k$-multilinear maps, respectively. All these spaces are by default endowed with the topology of uniform convergence on bounded sets (\cite[Section 8.4]{Jarchow}). By $\csn \bE$ we denote the set of continuous seminorms on $\bE$.

Given a (commutative) ring $R$ and two $R$-modules $A$ and $B$, we denote by $\Hom_R(A,B)$ the set of $R$-module homomorphisms from $A$ to $B$. Our references for module theory are \cite{Adkins,Bourbaki}.

For differential calculus on arbitrary locally convex spaces we employ the \emph{convenient setting} of \cite{KM}. $C^\infty(\bE, \bF)$ is the space of smooth mappings in this sense. If $\bE$ is finite dimensional this notion of smoothness is the classical one (cf.\ \cite[Chapter 40]{Treves}). We write $C^\infty(\bE)$ in place of $C^\infty(\bE, \bK)$. The main features we need of this calculus will be the exponential law $C^\infty(\bE_1 \times \bE_2, \bF) \cong C^\infty(\bE_1, C^\infty(\bE_2, \bF))$ and the differentiation operator $\ud\colon C^\infty(\bE, \bF) \to C^\infty(\bE, \Linbb(\bE, \bF))$.

In terms of differential geometry we will mainly follow \cite{Lang}. All manifolds will be assumed to be real, Hausdorff and paracompact. A vector bundle $E$ over a manifold $M$ with projection $\pi$ is denoted by $\pi\colon E \to M$ or simply $E \to M$. $\Gamma(E)$ denotes the space of smooth sections of $E$ and $\Gamma_c(E)$ the subset of those with compact support. For $x \in M$, $E_x = \pi^{-1}(x)$ denotes the fiber of $E$ over $x$. For any open subset $U \subseteq M$, $\Gamma(U, E)$ and $\Gamma_c(U, E)$ denote the respective sets of sections over $U$. For $f \in C^\infty(M)$, $m_f \colon \Gamma(E) \to \Gamma(E)$ denotes fiberwise multiplication by $f$. Given an atlas $(U_i, \varphi_i)_i$ of $M$, $\Gamma(U, E)$ is endowed with the projective topology with respect to the mappings $(\varphi_i^{-1})^*$ into the spaces $C^\infty(\varphi_i(U_i), \bE)$ containing the local representations of sections, where $\bE$ is the typical fiber of $E$. $\Gamma_c(U, E)$ is endowed with the corresponding (LF)-topology.
If $U \subseteq M$ is open and $\tau \colon \pi^{-1}(U) \to U \times \bK^{\dim E}$ trivializing, a section $s \in \Gamma(E)$ has local coordinates $s^i \in C^\infty(U)$ for $i=1\dotsc \dim E$. Given two vector bundles $\pi \colon E \to M$ and $\pi' \colon E' \to M'$ and a pair of morphisms $f_0 \colon M \to M'$ and $f \colon E \to E'$ such that $\pi' \circ f = f_0 \circ \pi$ and $f$ is fiberwise linear, we call $f$ a vector bundle morphism over $f_0$. 
By $f^* \colon \Gamma(E') \to \Gamma(E)$ we denote the pullback of sections and if $f$ is a diffeomorphism we set $f_* \coleq (f^{-1})^*$; this notation will also be used for the pullback of distributions. $\fX(M)$ denotes the vector space of smooth vector fields on $M$ and for $X \in \fX(M)$, $\Fl^X_t$ denotes the flow of $X$ at time $t$. Given a Riemannian metric $g$ on a manifold $M$, $B^g_r(x)$ denotes the metric ball of radius $r$ around $x \in M$ and $d_g (x,y)$ the induced Riemannian distance between two points $x,y \in M$ induced by $g$.

For basic notions of sheaf theory we refer to \cite{zbMATH03489364,MR0345092}.

Our references for distribution theory are \cite{TD,zbMATH03230708} for the local case and \cite{GKOS} for distributions on manifolds. Contrary to L.~Schwartz' we will write $C^\infty(\Omega)$ instead of $\cE(\Omega)$ for the space of smooth functions on an open subset $\Omega \subseteq \bR^n$; still, we denote the strong dual of $C^\infty(\Omega)$ by $\cE'(\Omega)$.

Let a manifold $M$ and a vector bundle $E \to M$ be given. With $\Vol(M)$ denoting the volume bundle of $M$, 
 (scalar) distributions on $M$ are defined to be elements of $\cD'(M) \coleq ( \Gamma_c(\Vol(M)) )' = \Lincb(\Gamma_c(\Vol(M)), \bK)$ and $E$-valued distributions are defined to be elements of $\cD'(M,E) \coleq (\Gamma_c(M, E^* \otimes \Vol(M)))'$ where $E^*$ is the dual bundle of $E$. Note the $C^\infty(M)$-module isomorphism $\cD'(M, E) \cong \Gamma(E) \otimes_{C^\infty(M)} \cD'(M)$ (cf.~\cite{vecreg,sectop}) 

These spaces of distributions are always endowed with the strong topology.

We denote the local coordinates of a distribution $u \in \cD'(M,E)$ by $u^i \in \cD'(U) \cong \cD'(\varphi(U))$ for $i=1\dotsc \dim E$.

\subsection{Smoothing operators}\label{subsec_smoothop}

The core idea of the functional analytic approach to Colombeau algebras rests on a variant of L.~Schwartz' kernel theorem \cite[Th\'eor\`eme 3]{FDVV}, namely the topological vector space isomorphism (for any open set $\Omega \subseteq \bR^n$)
\begin{equation}\label{kthm}
\SO \Omega \coleq \Lincb(\cD'(\Omega), C^\infty(\Omega)) \cong C^\infty(\Omega, \cD(\Omega)) \eqcol \SK \Omega.
\end{equation}
This isomorphism expresses the fact that distributions are regularized in a reasonable (i.e., linear and continuous) way \emph{exactly} by applying them to elements of $C^\infty(\Omega, \cD(\Omega))$, the latter space carrying the topology of uniform convergence on compact sets in all derivatives (\cite[Definition 40.2]{Treves}). Explicitly, the correspondence between \emph{smoothing operators} $\Phi \in \SO \Omega$ and \emph{smoothing kernels} $\vec\varphi \in \SK \Omega$ is given by $(\Phi u)(x) \coleq \langle u, \vec\varphi(x) \rangle$ and $\vec\varphi(x) \coleq \Phi^t (\delta_x)$ for $u \in \cD'(\Omega)$ and $x \in \Omega$,
where $\Phi^t \in \Lincb(\cE'(\Omega), \cD(\Omega))$ is the transpose of $\Phi$. The advantage of using smoothing operators instead of smoothing kernels is that this notion easily extends to vector valued distributions:

\begin{definition}
 Let $\pi\colon E \to M$ be a vector bundle and $U \subseteq M$ open. Then we call
\[ \VSO{U, E} \coleq \Lincb(\cD'(U,E), \Gamma(U,E)) \]
 the space of \emph{vector smoothing operators} on $E$ over $U$. We write $\VSO{E}$ instead of $\VSO{M, E}$. For $f \in C^\infty(U)$ and $\Phi \in \VSO{U,E}$ we set $f \cdot \Phi \coleq m_f \circ \Phi$, which turns $\VSO{U,E}$ into a $C^\infty(U)$-module.
\end{definition}
We remark that a representation similar to \eqref{kthm} can be obtained for vector smoothing operators (cf.~\cite{vecreg}) as
\begin{equation}\label{vecblahiso}
\VSO{E} \cong \Gamma(E^* \boxtimes E) \opotimes_{\mathclap{C^\infty(M \times M)}} C^\infty(M, \Vol(M))
\end{equation}
where $\boxtimes$ denotes the external tensor product of vector bundles (\cite[Chapter II.5]{GHV}).
In hindsight, this isomorphism completely explains the role of \emph{transport operators} (i.e., elements of $\Gamma(E^* \boxtimes E)$), which in \cite{global2} have been combined with smoothing kernels $\Phi \in C^\infty(M, \Vol(M))$ on heuristical grounds in order to regularize $E$-valued distributions. We base our presentation on vector smoothing operators not only because the representation as a tensor product would introduce unnecessary complications for our general purpose, but also because the formalism of smoothing operators can immediately be adapted for other spaces of distributions which are not necessarily given by dual spaces of appropriate spaces of test functions (\cite{vecreg}).

Suppose $E$ is trivializable over an open subset $U \subseteq M$ and $\dim E = m$. Then $\Gamma(U, E) \cong C^\infty(U)^m$ and $\cD'(U, E) \cong \cD'(U)^m$, hence every $\Phi \in \VSO{U,E}$ is represented by an $m \times m$-matrix $(\Phi_{ij})_{i,j=1 \dotsc m}$ with $\Phi_{ij} \in \SO U$ such that $\Phi(u)^i = \sum_j \Phi_{ij} (u^j)$:
\[
 \xymatrix{
\cD'(U, E)	& \cD'(U)^m	& \cD'(U) \\
\Gamma(U, E)	& C^\infty(U)^m	& C^\infty(U)
\ar"1,1";"2,1"_{\Phi}
\ar@{}"1,1";"1,2"|-*[@]\txt{$\cong$}
 \ar"1,3";"1,2"_-{\iota_j}
\ar@{}"2,1";"2,2"|-*[@]\txt{$\cong$}
\ar"2,2";"2,3"^-{\pr_i}
\ar"1,3";"2,3"^{\Phi_{ij}}
}
\]
Here, $\iota_j$ and $\pr_i$ denote the canonical injection of the $j$th and projection to the $i$th component, respectively.

\section{Review of the scalar case}\label{sec_revscalar}

\def\GFa{{\notationcolor F}}

We first recall the construction of the scalar Colombeau algebra on an open subset $\Omega \subseteq \bR^n$ in the functional analytic setting from \cite{papernew} in order to introduce some fundamental concepts and set the stage for the global vector-valued case.

Colombeau's original construction starts with an obvious candidate for an algebra containing $\cD'(\Omega)$ and allowing for many nonlinear operations, namely $C^\infty(\cD(\Omega))$, the space of \emph{smooth} functions $\cD(\Omega) \to \bK$ (smoothness here is understood in the sense of \cite{KM}, while Colombeau's original construction was based on \cite{zbMATH03798426}).
One then forms a certain quotient of this algebra (or rather a subalgebra of it, to be precise) in which $C^\infty(\Omega)$ becomes a subalgebra. Denote by $\iota\colon \cD'(\Omega) \to C^\infty(\cD(\Omega))$ the canonical inclusion. Using reflexivity of $C^\infty(\Omega)$, which means that $C^\infty(\Omega) \cong \Lincb(\cE'(\Omega), \bK)$, one sees that $\iota$ actually maps $C^\infty(\Omega)$ into the subalgebra $C^\infty(\cE'(\Omega)) \subseteq C^\infty(\cD(\Omega))$\proofs{\#39}, where the latter inclusion is given by the restriction map (\cite[\S 3.1]{ColNew}).

Elements of the form $\iota(f)\iota(g) - \iota(fg)$ for $f,g \in C^\infty(\Omega)$, which one would like to vanish in a suitable quotient of $C^\infty(\cD(\Omega))$, evaluate to zero on any $\delta_x \in \cE'(\Omega)$, where $\delta_x$ denotes the Dirac delta distribution at $x \in \Omega$. This suggests to search for an ideal of $C^\infty(\cD(\Omega))$
containing the set
\[ \mathcal{K}(\Omega) \coleq \{\, \GFa \in C^\infty(\cE'(\Omega))\ |\ \GFa(\delta_x) = 0\ \forall x \in \Omega\,\}. \]
$\mathcal{K}(\Omega)$ itself is an ideal in $C^\infty(\cE'(\Omega))$.
Because $\GFa \in C^\infty(\cD(\Omega))$ cannot be evaluated at $\delta_x$, the property $\GFa(\delta_x)=0$ cannot be used directly for extending $\mathcal{K}(\Omega)$ to an ideal of $C^\infty(\cD(\Omega))$. The well-known classical scheme of construction of Colombeau algebras (see also \cite{MOBook,GKOS}) is based on characterizing elements $\GFa \in \mathcal{K}(\Omega)$ by evaluating them on scaled and translated test functions $\varphi_{\e,x} (y) \coleq \e^{-n} \varphi ( (y-x) / \e )$, where $\varphi \in \cD(\Omega)$ and $\e \in (0,1]$; one sees that those $\GFa \in C^\infty(\cE'(\Omega))$ which lie in $\mathcal{K}(\Omega)$ are characterized by the property that the more moments of $\varphi$ vanish, the faster $\GFa(\varphi_{\e,x})$ converges to $0$ uniformly for $x$ in compact sets when $\e \to 0$ (\cite[Proposition 3.3.3]{ColNew}).

In order to express this in functional analytic and coordinate-invariant terms we take note of the mappings $\vec \delta \in C^\infty(\Omega, \cE'(\Omega))$, $\vec\delta(x) \coleq \delta_x$ and $\widetilde \GFa \in C^\infty ( C^\infty(\Omega, \cE'(\Omega)), C^\infty(\Omega))$ defined by $\widetilde \GFa (\vec\varphi) (x) \coleq \GFa ( \vec\varphi(x) )$ for each $\GFa \in C^\infty(\cE'(\Omega))$. The property defining $\mathcal{K}(\Omega)$ can then be written as $\widetilde \GFa ( \vec \delta) = 0$. We emphasize that this is not merely a different notation but a change of viewpoint enabling one to formulate the following essential abstraction of \cite[Proposition 3.3.3]{ColNew}.

\begin{lemma}\label{colbasis}Let $\GFa \in C^\infty(\cE'(\Omega))$.
For any $\csna \in \csn { C^\infty(\Omega) }$ there exists $\csnb \in \csn { C^\infty(\Omega, \cE'(\Omega))}$ such that 
\[ \csna ( \widetilde \GFa (\vec\varphi) - \widetilde \GFa ( \vec \psi) ) \le \csnb ( \vec\varphi - \vec \psi ) \qquad \forall \vec\varphi,\vec\psi \in C^\infty(\Omega, \cE'(\Omega)). \]
In particular, if $\csna$ is given by $f \mapsto \sup_{x \in K, \abso{\alpha}\le m} \abso{(\pd^\alpha f)(x)}$ with $K \subseteq \Omega$ compact and $m \in \bN_0$, $\csnb$ is given by $\vec\varphi \mapsto \sup_{x \in K,\abso{\alpha} \le m} \csnc ( (\pd^\alpha \vec\varphi)(x) )$ for some $\csnc \in \csn { \cE'(\Omega) }$.
\end{lemma}
In other words, $\widetilde \GFa \colon C^\infty(\Omega, \cE'(\Omega)) \to C^\infty(\Omega)$ is uniformly continuous (although smooth functions in the sense of \cite{KM} need not even be continuous, in general). It follows that for $(\vec\varphi_\e)_\e \in \SK\Omega^{(0,1]}$ and $\GFa \in C^\infty(\Omega, \cE'(\Omega))$, $\vec\varphi_\e \to \vec\varphi$ in $C^\infty(\Omega, \cE'(\Omega))$ implies $\widetilde \GFa (\vec\varphi_\e) - \widetilde \GFa ( \vec\varphi) \to 0$ in $C^\infty(\Omega)$. Even more, $\csnb(\vec\varphi_\e - \vec\varphi) = O(\e^m)$ for all $\csnb$ and $m \in \bN$ implies $\csna(\widetilde \GFa(\vec\varphi_\e) - \widetilde \GFa(\vec\varphi)) = O(\e^m)$ for all $\csna$ and $m \in \bN$
(we call this \emph{rapid convergence}).
\begin{proof}\checkedproof
 By the mean value theorem (\cite[1.4]{KM}) and the chain rule (\cite[3.18]{KM}), $\widetilde \GFa (\vec\varphi) - \widetilde \GFa(\vec \psi)$ is contained in the closed convex hull of the set\proofs{\#40.1}
\[ \{\, \ud \widetilde \GFa(\vec\psi + t ( \vec\varphi - \vec\psi)) (\vec\varphi - \vec\psi)\ |\ t \in (0,1)\,\}. \]
Let $\csna$ be a continuous seminorm on $C^\infty(\Omega)$, which can be taken to be of the form $\csna(f)=\sup_{x \in K, \abso{\alpha}\le m} \abso{(\pd^\alpha f)(x)}$ with $K \subseteq \Omega$ compact and $m \in \bN_0$ (the family of all such seminorms forms a basis of continuous seminorms of $C^\infty(\Omega)$). For the claim to hold it suffices to show that there is $\csnb$ such that $\csna ( \ud \widetilde \GFa ( \vec\psi + t (\vec\varphi - \vec\psi))(\vec\varphi - \vec\psi)) \le \csnb (\vec\varphi - \vec\psi)$ for all $t \in (0,1)$.\proofs{\#40.2} We first note that for $x \in \Omega$ and $\alpha \in \bN_0^n$, $\pd^\alpha ( \ud \widetilde \GFa (\vec\psi + t (\vec\varphi - \vec\psi)) (\vec\varphi - \vec\psi))(x) = \pd_x^\alpha ( \ud \GFa ( \vec\psi(x) + t ( \vec\varphi(x) - \vec\psi(x))) (\vec\varphi(x) - \vec\psi(x)))$ is given, again by the chain rule, by a linear combination of terms of the form \proofs{\#40.3}
\begin{multline}\label{alphar}
 \ud^{k+1} \GFa \bigl(\vec\psi(x) + t (\vec\varphi(x) - \vec\psi(x))\bigr)\bigl(\pd^{\beta_1} ( \vec\psi + t (\vec\varphi - \vec\psi) )(x), \dotsc, \\ \pd^{\beta_{k}} ( \vec\psi + t (\vec\varphi - \vec\psi) )(x), \pd^{\beta_{k+1}} (\vec\varphi - \vec\psi)(x) \bigr)
\end{multline}
with $k \in \bN_0$ and multiindices $\beta_1,\dotsc,\beta_{k+1} \in \bN_0^n$ satisfying $k \le \abso{\alpha}$ and $\beta_1 + \dotsc + \beta_{k+1} = \alpha$. Because $(t,x) \mapsto \ud^{k+1} \GFa ( \vec\psi(x) + t (\vec\varphi(x) - \vec\psi(x)) )$ is smooth (and hence continuous) from $\bR \times \Omega$ into $\Linbbmult{k+1}(\cE'(\Omega)^{k+1}, \bK)$, it maps $[0,1] \times K$ into a bounded subset of
$\Linbbmult{k+1}( \cE'(\Omega)^{k+1}, \bK)$. By the exponential law for bounded linear mappings \cite[5.2]{KM} this space is bornologically isomorphic to $\Linbbmult{k}( \cE'(\Omega)^{k}, \Linbb(\cE'(\Omega), \bK))$ and because each $\{\ \pd^{\beta_i} ( \vec\psi + t ( \vec\varphi - \vec\psi ) ) (x)\ |\ t \in [0,1],\ x \in K\ \}$ ($i=1,\dotsc,k$) is bounded in $\cE'(\Omega)$ the set
\begin{multline*}
\{\,\ud^{k+1} \GFa (\vec\psi(x) + t ( \vec\varphi(x) - \vec\psi(x))) ( \pd^{\beta_1} ( \vec\psi + t ( \vec\varphi - \vec\psi ) ) (x), \dotsc, \\
\pd^{\beta_{k}} ( \vec\psi + t ( \vec\varphi - \vec\psi ) ) (x) )\ |\ t \in [0,1],\ x \in K\, \}
\end{multline*}
is bounded and hence equicontinuous in $\Linbb(\cE'(\Omega), \bK)$ because $\cE'(\Omega)$ is bornological and barrelled.
Hence, there is a continuous seminorm $\csnc$ of $\cE'(\Omega)$ such that \eqref{alphar} can be estimated by $\csnb ( \vec\varphi - \vec \psi) \coleq \sup_{x \in K, \abso{\alpha} \le m} \csnc (\pd^\alpha ( \vec\varphi - \vec\psi) (x) )$, which is a continuous seminorm of $C^\infty(\Omega, \cE'(\Omega))$.
\end{proof}

Consequently, $\widetilde \GFa(\vec\delta) = 0$ if and only if $\widetilde \GFa(\vec\varphi_\e) \to 0$ in $C^\infty(\Omega)$ for any net $(\vec\varphi_\e)_\e$ in $C^\infty(\Omega, \cE'(\Omega))$ with $\vec\varphi_\e \to \vec\delta$. Although the latter condition makes sense also for $\GFa \in C^\infty(\cD(\Omega))$ by defining $\widetilde \GFa(\vec\varphi)(x) \coleq \GFa(\vec\varphi(x))$ for $\vec\varphi \in \SK \Omega$, the corresponding subspace
\[ \{\, \GFa \in C^\infty(\cD(\Omega))\ |\ \widetilde \GFa(\vec\varphi_\e) \to 0 \text{ for every net }(\vec\varphi_\e)_\e \text{ in }\SK\Omega\text{ with } \vec\varphi_\e\to \vec\delta\,\}\]
is no ideal in $C^\infty(\cD(\Omega))$: taking, e.g., $u \in \cE'(\Omega)$, $K \subseteq \Omega$ compact and $m \in \bN_0$, the expression
\begin{equation}\label{blahgux}
\sup_{x \in K, \abso{\alpha} \le m} \abso{ \pd^\alpha \widetilde{(\iota u)}(\vec\varphi_\e ) (x)} = \sup_{\mathclap{x \in K, \abso{\alpha} \le m}} \abso{\langle u, (\pd^\alpha \vec\varphi_\e)(x) \rangle } \le C \cdot \sup_{\mathclap{\substack{x \in K,y \in \Omega\\\abso{\alpha}\le m, \abso{\beta} \le l}}} \abso{(\pd_x^\alpha \pd_y^\beta \vec\varphi_\e)(x)(y)}
\end{equation}
(where $l$ is the order of $u$) can diverge faster than the convergence of $\widetilde \GFa(\vec\varphi_\e) \to 0$ takes place. However, Lemma \ref{colbasis} shows that for $\GFa \in \mathcal{K}(\Omega)$ we can make $\widetilde \GFa(\vec\varphi_\e) \to 0$ converge rapidly by taking $\vec\varphi_\e \to \vec\delta$ rapidly, so we only need to have a polynomial bound in $1/\e$ of \eqref{blahgux} in order to make
\[ \{\, \GFa \in C^\infty(\cD(\Omega))\ |\ \widetilde \GFa(\vec\varphi_\e) \to 0 \text{ rapidly for }\vec\varphi_\e \to \vec\delta \text{ rapidly}\, \} \]
invariant under multiplication by elements of $\iota(\cD'(\Omega))$. It cannot be stable under multiplication by arbitrary elements $\GFa \in C^\infty(\cD(\Omega))$, therefore we restrict to the space of those $\GFa$ such that $\widetilde \GFa(\vec\varphi_\e)$ grows at most polynomially in $1/\e$, mimicking the behaviour of distributions above.

In order to ensure that nonzero distributions will not lie in the ideal and also because intuitively, $\widetilde{(\iota u)} (\vec\varphi_\e)$ should be seen as an approximation of $u \in \cD'(\Omega)$, we require that $\langle u, \vec\varphi_\e \rangle \to u$ in $\cD'(\Omega)$ (this ensures that the concept of association, which is fundamental to Colombeau algebras in order to obtain coherence with classical analysis, is available). Interpreting each $\vec\varphi$ as a regular $\cD(\Omega)$-valued distribution, i.e., an element of $\Lincb(\cD(\Omega), \cD(\Omega))$, via the vector-valued integral $\langle \vec\varphi, \psi \rangle \coleq \int \vec\varphi(x) \psi(x)\,\ud x \in \cD(\Omega)$ for $\psi \in \cD(\Omega)$, this is easily seen to be equivalent to requiring $\vec\varphi_\e \to \id$ in $\Lincb(\cD(\Omega), \cD(\Omega))$.

These considerations lead to the use of so-called \emph{test objects}, which are nets $(\vec\varphi_\e)_\e$ in $\SK\Omega$ satisfying
\begin{enumerate}[label=(\arabic*)]
 \item \label{first} $\forall \csna \in \csn { C^\infty(\Omega, \cE'(\Omega)) }\ \forall m \in \bN: \csna (\vec\varphi_\e - \vec\delta) = O(\e^m)$;
 \item \label{second} $\vec\varphi_\e \to \id$ in $\Lincb(\cD(\Omega), \cD(\Omega))$;
 \item \label{third} $\forall \csna \in \csn { C^\infty(\Omega, \cD(\Omega))}$ $\exists N \in \bN$: $\csna(\vec\varphi_\e) = O(\e^{-N})$.
\end{enumerate}

From the above discussion it appears that the crucial objects of the quotient construction are not elements of $C^\infty(\cD(\Omega))$ but of $C^\infty( \SK \Omega, C^\infty(\Omega))$. Hence,
we are led to think of generalized functions as 
\emph{mappings from smoothing kernels to smooth functions}, which gives a much greater flexibility which will be of significant use later on. We reach the following definitions:
\begin{definition}
 We set $\Gnb\Omega \coleq C^\infty ( \SK\Omega, C^\infty(\Omega))$, with embeddings $\iota\colon \cD'(\Omega) \to \Gnb\Omega$ and $\sigma \colon C^\infty(\Omega) \to \Gnb\Omega$ defined by $(\iota u)(\vec\varphi)(x) \coleq \langle u, \vec\varphi(x) \rangle$ and $(\sigma f)(\vec\varphi) \coleq f$. Given a diffeomorphism $\mu \colon \Omega \to \Omega'$, its action $\mu_*\colon \Gnb\Omega \to \Gnb{\Omega'}$ is defined as $(\mu_* \GFa)(\vec\varphi) \coleq \mu_* ( \GFa ( \mu^*\vec\varphi)) = \GFa(\mu^*\vec\varphi) \circ \mu^{-1}$,
where $(\mu^* \vec\varphi)(x) \coleq (\vec\varphi(\mu x) \circ \mu) \cdot \abso{\det \D\mu}$ is the natural pullback of smoothing kernels. The Lie derivative of $F \in \Gnb\Omega$ with respect to a vector field $X \in C^\infty(\Omega, \bK^n)$ is defined as $(\Lie_X \GFa) (\vec\varphi) \coleq  - (\ud \GFa) (\vec\varphi) (\Lsk_X \vec\varphi) + \Lie_X ( \GFa ( \vec\varphi))$, where $(\Lsk_X\vec\varphi)(x) = \Lie_X ( \vec\varphi (x)) + (\Lie_X\vec\varphi)(x)$ is the Lie derivative of smoothing kernels.
\end{definition}
Note that this action of diffeomorphisms is the natural one and the Lie derivative $\Lie_X$ is obtained by differentiating the pullback along the flow of $X$.

For the quotient construction, let the space $\TO \Omega$ of \emph{test objects} on $\Omega$ be given by all nets $(\vec\varphi_\e)_\e \in \SK \Omega^{(0,1]}$ having properties (1)--(3) above, and $\TOz \Omega$ its parallel vector subspace obtained by replacing convergence to $\vec\delta$ in \ref{first} and to $\id$ in \ref{second} by convergence to $0$. Note that for $(\vec\varphi_\e)_\e \in \TO\Omega$, $(\Lsk_X\vec\varphi_\e)_\e$ is an element of $\TOz \Omega$. 
\begin{definition}An element $\GFa \in \Gnb\Omega$ is called \emph{moderate} if $\forall \csna \in \csn {C^\infty(\Omega)}$ $\forall k \in \bN_0$ $\exists N \in \bN$ $\forall (\vec\varphi_\e)_\e \in \TO\Omega$, $(\vec\psi_{1,\e})_\e \dotsc (\vec\psi_{k,\e})_\e \in \TOz\Omega$:
\[ \csna \left(\, (\ud^k \GFa) (\vec\varphi_\e)(\vec\psi_{1,\e},\dotsc,\vec\psi_{k,\e})\, \right) = O(\e^{-N}). \]
The set of all moderate elements of $\Gnb\Omega$ is denoted by $\Gnm\Omega$. 
An element $\GFa \in \Gnb\Omega$ is called \emph{negligible} if $\forall \csna \in \csn {C^\infty(\Omega)}$ $\forall k \in \bN_0$ $\forall m \in \bN$ $\forall (\vec\varphi_\e)_\e \in \TO\Omega$, $(\vec\psi_{1,\e})_\e \dotsc (\vec\psi_{k,\e})_\e \in \TOz\Omega$:
\[ \csna \left(\, (\ud^k \GFa) (\vec\varphi_\e)(\vec\psi_{1,\e},\dotsc,\vec\psi_{k,\e})\, \right) = O(\e^m). \]
The set of all negligible elements of $\Gnb\Omega$ is denoted by $\Gnn\Omega$.
We set $\Gen\Omega \coleq \Gnm\Omega / \Gnn\Omega$.
\end{definition}
We recall that $\iota$ and $\sigma$ map into $\Gnm\Omega$ and commute with diffeomorphisms and Lie derivatives, $\iota|_{C^\infty(\Omega)} - \sigma$ maps into $\Gnn\Omega$, $\iota$ is injective into $\Gen\Omega$, and sums, products, diffeomorphisms and Lie derivatives preserve moderateness and negligibility and hence are well-defined on $\Gen\Omega$. Furthermore, we note that in order to obtain a sheaf one has to require, in addition to \ref{first}--\ref{third}, that $\supp \vec\varphi_\e(x)$ converges to $x$ in a certain sense. For further details we refer to \cite{papernew}.

\begin{remark}\label{rema}
\begin{enumerate}[label=(\roman*)]
 \item \label{rema.1} Because of isomorphism \eqref{kthm} one can also formulate this construction in terms of smoothing operators $\Phi \in \SO \Omega$, using the basic space $\Gnb \Omega \coleq C^\infty(\SO \Omega, C^\infty(\Omega))$ and corresponding conditions on nets $(\Phi_\e)_\e \in \SO \Omega^{(0,1]}$. Then, the Colombeau-product of two embedded distributions, $(\iota(u) \cdot \iota(v))(\vec\varphi) \coleq \langle u, \vec\varphi \rangle \cdot \langle v, \vec\varphi \rangle$, turns into $(\iota(u) \cdot \iota(v))(\Phi) \coleq \Phi(u) \cdot \Phi(v)$ for $\Phi \in \SO \Omega$, which is nothing else than taking the usual product of the \emph{regularizations} of $u$ and $v$. This idea will also form the basis of the definition of the tensor product of generalized vector fields.
\item This method of obtaining a diffeomorphism invariant Colombeau algebra is considerably simpler than the previous approach in \cite{found,global}, see Remark \ref{remarkiblah} \ref{remarkiblah.3}.
\end{enumerate}

\end{remark}

\section{The basic space}

\def\GSa{{\notationcolor  R}}
\def\GSb{{\notationcolor S}}
\def\SSa{{\notationcolor r}}
\def\SSb{{\notationcolor s}}
\def\SDa{{\notationcolor \alpha}}
\def\SDb{{\notationcolor \beta}}
\def\GDa{{\notationcolor \alpha}}
\def\GDb{{\notationcolor \beta}}
\def\SMa{{\notationcolor \mu}}
\def\GVa{{\notationcolor X}}
\def\GVb{{\notationcolor Y}}
\def\GFb{{\notationcolor G}}

In this section we will introduce the spaces containing the representatives of generalized sections of vector bundles and define the basic operations on them.

The strength of the functional analytic approach outlined in Section 3 is that it transfers directly to the setting on manifolds and also to vector bundle valued distributions simply by replacing $\SK \Omega \cong \SO \Omega$ by the appropriate space of \emph{vector smoothing operators}. However, one point of fundamental importance has to be made clear: if one accepts the premise that, conceptually, generalized functions are best seen as functions depending on smoothing operators (which provide the embedding of distributions), this means that in the vector valued case, distributions taking values in different vector bundles will have to depend on different smoothing operators -- as long as no additional structure is introduced which relates smoothing operators of different vector bundles.

In a first step, given a vector bundle $E \to M$ we note that distributions $u \in \cD'(M, E)$ act on vector smoothing operators $\Phi \in \VSO{E}$ via the canonical embedding
\begin{align*}
\iota\colon \cD'(M,E) & \to \Lincb(\VSO{E}, \Gamma(E)) \\
u &\mapsto [\Phi \mapsto \Phi(u)]
\end{align*}
which corresponds to the mapping $R \mapsto \widetilde R$ of Section \ref{sec_revscalar} (but with linear maps). Denoting the image of $u$ under this embedding by the same letter, we henceforth write $u(\Phi) \coleq \Phi(u)$.

Following the scheme of the scalar case (see Remark \ref{rema} \ref{rema.1} above), we will extend the usual operations on smooth sections of $E$ (vector space structure, tensor product, permutation, contraction with dual tensor fields, derivatives etc.) to distributional sections $u \in \cD'(M, E)$ \emph{by applying them to the regularizations} $u(\Phi) \in \Gamma(E)$, where $\Phi \in \VSO E$ is a parameter on which the generalized section depends, similar to the role test functions have in distribution theory. This results in a smooth section depending on the same vector smoothing operators as the arguments of the operation. For example, the tensor product of $u \in \cD'(M, E)$ and $v \in \cD'(M, F)$, where $F \to M$ is another vector bundle, shall be given by the bilinear continuous (or, more generally, \emph{smooth}) mapping
\begin{equation}\label{genmult}
\begin{aligned}
 u \otimes v \colon \VSO E \times \VSO F &\to \Gamma(E \otimes F), \\
 (\Phi, \Psi) & \mapsto u(\Phi) \otimes_{C^\infty(M)} v(\Psi).
\end{aligned}
\end{equation}
In the scalar case (i.e., for $M=\Omega \subseteq \bR^n$ and $E = F = \Omega \times \bK$), requiring commutativity of this product for $u,v \in \cD'(\Omega)$ would force the mapping $u \cdot v \colon (\Phi, \Psi) \mapsto u(\Phi) \cdot v(\Psi)$ to be symmetric, which naturally leads to taking polynomials or consequently smooth functions on $\SO\Omega$ as the basic space and in turn gives the construction outlined in Section \ref{sec_revscalar}. In the general case, however, the tensor product has to satisfy the following natural conditions:

\begin{enumerate}[label=(\roman*)]
\item \label{cond2} For $E=F$, the permutation $\GSa \otimes \GSb \to \GSb \otimes \GSa$ of the tensor product \eqref{genmult} should be compatible with the permutation $p\colon \SSa \otimes \SSb \mapsto \SSb \otimes \SSa$ on $\Gamma(E) \otimes_{C^\infty(M)} \Gamma(E)$ in the sense that $p \circ (\GSa \otimes \GSb) = \GSb \otimes \GSa$ as elements of $C^\infty(\VSO E \times \VSO E, \Gamma(E \otimes E))$. This implies that $\GSa(\Phi_1) \otimes \GSb(\Phi_2) = \GSa(\Phi_2) \otimes \GSb(\Phi_1)$ for $\Phi_1,\Phi_2 \in \VSO E$, which forces us to set $\Phi_1 = \Phi_2$. Hence, distributions with values in the same vector bundle have to be regularized by the same vector smoothing operators.
 \item \label{cond3} Contraction should be compatible with permutations in the following sense. Let $\GSa \in C^\infty(\VSO E, \Gamma(E))$, $\GSb \in C^\infty(\VSO F, \Gamma(F))$, $\SDa \in \Gamma(E^*)$ and $\SDb \in \Gamma(F^*)$. Defining $\GSa \cdot \SDa \in C^\infty(\VSO E, C^\infty(M))$ by $(\GSa \cdot \SDa)(\Phi) \coleq \GSa(\Phi) \cdot \SDa$, we require that $(\GSa \otimes \GSb)(\SDa \otimes \SDb) = (\GSb \otimes \GSa)(\SDb \otimes \SDa)$. However, using \eqref{genmult}, $(\GSa \otimes \GSb)(\SDa \otimes \SDb)$ lies in $C^\infty(\VSO E \times \VSO F, C^\infty(M))$ while $(\GSb \otimes \GSa)(\SDb \otimes \SDa)$ is an element of $C^\infty(\VSO F \times \VSO E, C^\infty(M))$. The necessary identification of $\VSO E \times \VSO F$ and $\VSO F \times \VSO E$ will be taken care of by using the proper notation.
\end{enumerate}

From this it follows that $\VSO E$ should appear at most \emph{once} for each vector bundle $E$ as a parameter space of a generalized tensor field. Moreover, one should not impose any specific order on arguments of generalized sections. Formally, this is accomplished by the following definition.

\begin{definition}\label{def_basicspace}
\begin{enumerate}[label=(\roman*)]
 \item Let $E \to M$ be a vector bundle and $\Delta$ a (possibly empty) finite set of vector bundles. Then we define
\begin{align*}
 \VSO \Delta & \coleq \prod_{G \in \Delta} \VSO G, \\
\Gb E\Delta &\coleq C^\infty(\VSO \Delta, \Gamma(E)).
\end{align*}
Elements of the vector space $\Gb E\Delta$ are called \emph{generalized sections of $E$ (with index set $\Delta$}). We call each $\Gb E\Delta$ a \emph{basic space}.
\item If $\Delta$ is the empty set the product becomes the trivial vector space $\{0\}$ and we identify $\Gb E\emptyset = C^\infty(\{0\}, \Gamma(E))$ with $\Gamma(E)$.
\item Distributions are embedded into $\Gb E{\{E\}}$ via $\iota \colon \cD'(M, E) \to \Gb E{\{ E \}}$, $(\iota u)(\Phi) \coleq \Phi (u)$ for $u \in \cD'(M,E)$ and $\Phi \in \VSO E$.
\item For any open subset $U \subseteq M$, the space of generalized sections of $E$ over $U$ is defined to be $\Gb{U, E}\Delta \coleq \Gb{E|_U}{(G|_U)_{G \in \Delta}}$. Furthermore, we write $\Gb M\Delta \coleq \Gb{M \times \bK}\Delta$. Elements of $\Gb M\Delta$ are called \emph{generalized scalar functions (with index set $\Delta$)}.
\end{enumerate}
\end{definition}

For shorter notation we will frequently write $\Phi \in \VSO \Delta$ instead of $(\Phi_G)_G \in \VSO \Delta$. $\Gb E\Delta$ is a $C^\infty(M)$-module with multiplication $(f \cdot R)(\Phi) \coleq f \cdot R(\Phi)$ for $f \in C^\infty(M)$, $R \in \Gb E\Delta$ and $\Phi \in \VSO \Delta$.
Because for $\Delta_1 \subseteq \Delta_2$ we have $\Gb E{\Delta_1} \subseteq \Gb E{\Delta_2}$, we can (and usually will) assume that all generalized sections we are dealing with have the same index set $\Delta$. We will consider $\iota$ from above as a map into $\Gb E{\{E\}}$.

\begin{definition}\label{basop}

Given $\GSa \in \Gb E\Delta$ and $\GSb \in \Gb F\Delta$ we define their tensor product $\GSa \otimes \GSb$ as the element of $\Gb E\Delta$ given by
\begin{equation}\label{genmultdef}
(\GSa \otimes \GSb) ( \Phi ) \coleq \GSa ( \Phi ) \otimes_{C^\infty(M)} \GSb ( \Phi ) \qquad (\Phi \in \VSO \Delta).
\end{equation}
Note that $(\GSa, \GSb) \mapsto \GSa \otimes \GSb$ is $C^\infty(M)$-bilinear.
\end{definition}

\begin{remark}
\begin{enumerate}[label=(\roman*)]
 \item These definitions comply with conditions \ref{cond2} and \ref{cond3} above.
\item For $E = F = M \times \bK$ this tensor product turns $\Gb M\Delta$ into an algebra.
\item The most striking feature of the general vector valued case is the fact that, compared to the scalar case, there is not only one but many basic spaces even for scalar generalized functions: contracting $\GSa \in C^\infty(\VSO E, \Gamma(E))$ with a dual tensor field $\SDa \in \Gamma(E^*)$ gives an element of $C^\infty(\VSO E, C^\infty(M))$, i.e., for any vector bundle $E$ we have scalar generalized functions in $\Gb M{\{E\}}$. Although at this level of generality there is no relation between these basic spaces, such may be obtained by imposing further structure as seen for example in \cite{global2}, where only tensor bundles are considered; see also Section \ref{to_classes}.
\end{enumerate}
\end{remark}

The next crucial result states that generalized sections in $\Gb E\Delta$ can be viewed as sections in $\Gamma(E)$ with coefficients in $\Gb M\Delta$, a result which holds analogously for distributions (\cite[Theorem 3.1.12]{GKOS}) and in $\hat\cG^r_s$ (\cite[Theorem 8.19]{global2}).

\begin{theorem}\label{reprmod}
$\Gb E\Delta \cong \Gb M\Delta \opotimes_{C^\infty(M)} \Gamma(E) \cong \Hom_{C^\infty(M)}(\Gamma(E^*), \Gb M\Delta)$ as $C^\infty(M)$-modules.
\end{theorem}
\begin{proof}
Using the fact that $\Gamma(E)$ is a projective $C^\infty(M)$-module the claim is easily reduced to the case $E = M \times \bK$ for which it holds trivially; see also \cite[Theorem 8.19]{global2} for a different proof which also applies here.
\end{proof}

Denoting the image of $\GSa \in \Gb E\Delta$ in $\Hom_{C^\infty(M)} ( \Gamma(E^*), \Gb M\Delta)$ also by $\GSa$, this isomorphism explicitly reads $(\GSa \cdot \SDa)(\Phi) = \GSa(\Phi) \cdot \SDa$ for $\Phi \in \VSO \Delta$ and $\SDa \in \Gamma(E^*)$. If $E$ is trivial and $(b_i)_{i=1\dotsc m}$ is a basis of $\Gamma(E)$ with dual basis $( \beta^i )_{i=1 \dotsc m}$ of $\Gamma(E^*)$ then $\GSa$ can be written as $\GSa = \sum_{i=1}^m \GSa^i b_i$ with coordinates $\GSa^i = \GSa \cdot \beta^i \in \Gb M\Delta$.

The following algebraic consequences of Theorem \ref{reprmod} follow immediately from standard module theory (\cite[Chapter II \S 5]{Bourbaki}):
\begin{corollary}\label{algprop}Let $E,F$ be vector bundles over $M$. Then
 \begin{enumerate}[label=(\roman*)]
  \item \label{algprop.1} $\Gb E\Delta$ is an $\Gb M\Delta$-module,
  \item if $\Gamma(E)$ has basis $(b_i)_i$, $\Gb E\Delta$ has basis $(1 \otimes b_i)_i$,
  \item $\Gb E\Delta$ is a projective $\Gb M\Delta$-module.
\end{enumerate}
Furthermore, the following $\Gb M\Delta$-module isomorphisms hold:
\begin{enumerate}[resume,label=(\roman*)]
  \item \label{algprop.4} $\Gb E\Delta \otimes_{\Gb M\Delta} \Gb F\Delta \cong \Gb {E \otimes F}\Delta$,
  \item \label{algprop.5}$\Hom_{C^\infty(M)} ( \Gamma(E), \Gamma(F) ) \otimes_{C^\infty(M)} \Gb M\Delta \cong \Hom_{\Gb M\Delta} ( \Gb E\Delta, \Gb F\Delta)$,
  \item \label{algprop.6}$\Hom_{C^\infty(M)} ( \Gamma(E), \Gb M\Delta) \cong \Hom_{\Gb M\Delta}(\Gb E\Delta, \Gb M\Delta)$.
 \end{enumerate}
 \end{corollary}
These properties are very useful in practice because they transfer classical isomorphisms to the generalized setting. Their explicit form is as follows. \ref{algprop.1}: for $\GSa \in \Gb E\Delta$ and $\GFa \in \Gb M\Delta$, $\GSa \cdot \GFa \in \Gb E\Delta$ is given by $(\GSa \cdot \GFa)(\Phi) = \GSa(\Phi) \cdot \GFa(\Phi)$ for $\Phi \in \VSO\Delta$, which is consistent with the product given by Definition \ref{basop}. Moreover, the $C^\infty(M)$-module structure on $\Gb E\Delta$ of Definition \ref{basop} is exactly the one obtained by restricting the ring of scalars of $\Gb E\Delta$ from $\Gb M\Delta$ to $C^\infty(M)$ via the embedding $C^\infty(M) = \Gb M\emptyset \subseteq \Gb M\Delta$. Similarly, $C^\infty(M)$ is a subalgebra of $\Gb M\Delta$. \ref{algprop.5}: for $h \in \Hom_{C^\infty(M)} ( \Gamma(E), \Gamma(F))$ and $\GFa \in \Gb M \Delta$, $h \otimes \GFa$ corresponds to the map $H \in \Hom_{\Gb M\Delta} ( \Gb E \Delta, \Gb F\Delta)$ defined by $H(\GSa)(\Phi) \coleq h(\GSa(\Phi)) \cdot \GFa(\Phi)$ for $\GSa \in \Gb E\Delta$ and $\Phi \in \VSO \Delta$. \ref{algprop.6}: for $h \in \Hom_{C^\infty(M)} ( \Gamma(E), \Gb M \Delta)$ and $\Phi \in \VSO \Delta$, the element $H \in \Hom_{\Gb M \Delta} ( \Gb E \Delta, \Gb F \Delta)$ corresponding to $h$ is given by $H(\GSa)(\Phi) \coleq h(\GSa(\Phi))(\Phi)$ for $\GSa \in \Gb E \Delta$ and $\Phi \in \VSO \Delta$.

Furthermore, we define the contraction of $\GSa \in \Gb E\Delta$ and $\GSb \in \Gb {E^*}\Delta$ in the obvious way by $(\GSa \cdot \GSb) \coleq \GSa(\Phi) \cdot \GSb(\Phi)$.

We will now consider the \emph{mixed tensor algebra}
\[
 \Gtb E\Delta \coleq \bigoplus_{r,s \ge 0} \Gb{E^r_s}\Delta
\]
where $E^r_s = E \otimes \dotsc E \otimes E^* \otimes \dotsc \otimes E^*$ with $r$ copies of $E$ and $s$ copies of $E^*$, and $E^0_0 = M \times \bK$. 

By Corollary \ref{algprop} \ref{algprop.5} it follows that if $r,s \ge 1$ then for all $(i,j)$ with $1 \le i \le r$, $1 \le j \le s$ there is a unique $\Gb M\Delta$-linear mapping $C^i_j \colon \Gb{E^r_s}\Delta \to \Gb{E^{r-1}_{s-1}}\Delta$, called \emph{$(i,j)$-contraction}, such that $C^i_j (\GSa_1 \otimes \dotsc \otimes \GSa_r \otimes \GDa^1 \otimes \dotsc \otimes \GDa^s)$ is given by
\[  (\GSa_i \cdot \GDa^j) \cdot \GSa_1 \otimes \dotsc \otimes \GSa_{i-1} \otimes \GSa_{i+1} \otimes \dotsc \otimes \GSa_r \otimes \GDa^1 \otimes \dotsc \otimes \GDa^{j-1} \otimes \GDa^{j+1} \otimes \dotsc \otimes \GDa^s  \]
for all $\GSa_1, \dotsc \GSa_r \in \Gb E\Delta$ and $\GDa^1,\dotsc,\GDa^s \in \Gb{E^*}\Delta$. In fact, this mapping is simply given by componentwise contraction of smooth tensor fields, i.e., we have $C^i_j(\GSa)(\Phi) = C^i_j(\GSa(\Phi))$ for $\GSa \in \Gb{E^r_s}\Delta$ if we denote by $C^i_j$ also the classical contraction.

A \emph{derivation} on $\Gtb E\Delta$ is a family of $\bK$-linear functions
\[ \cD = \cD^r_s \colon \Gb{E^r_s}\Delta \to \Gb{E^r_s}\Delta \qquad (r,s \ge 0) \]
such that $\cD ( \GSa \otimes \GSb ) = \cD \GSa \otimes \GSb + \GSa \otimes \cD \GSb$ and $\cD ( C^i_j \GSa ) = C^i_j ( \cD \GSa )$, i.e., $\cD$ satisfies the Leibniz rule and commutes with all contractions. Such a derivation is uniquely determined by its values on $\Gb M\Delta$ and $\Gb E\Delta$.

\subsection{Functoriality}

It was an open question for a long time whether a functorial construction of Colombeau's algebra of generalized functions is possible in the sense that any diffeomorphism $\Omega \to \Omega'$ between open subsets of $\bR^n$ (or, more generally, manifolds) induces a corresponding map between the respective Colombeau algebras. A complete answer was given for the first time in \cite{found}, based on previous work of several authors (\cite{colmani, Jelinek}). This, in turn, led to the construction of an intrinsic variant of Colombeau algebras on manifolds in \cite{global}. However, this construction was technically very involved (cf.~\cite[Section 2.1]{GKOS}).

On the level of the basic space this question is solved easily; in our setting the definition of vector smoothing operators and the basic space are functorial. However, our approach is distinguished by the fact that also the spaces of test objects (see Section \ref{sec_testobjects}) will be functorial by definition; this is a stark contrast to the situation of \cite{found}, where diffeomorphism invariance was only achieved by way of complicated modifications of the respective spaces of test objects.

Vector bundle isomorphisms act naturally on vector smoothing operators and the basic space as follows:

\begin{definition}\label{def_vbiso}
Let $\mu\colon E \to F$ be a vector bundle isomorphism over a diffeomorphism $f\colon M \to N$. We define the push-forward of $\Phi \in \VSO E$ along $\mu$ as the element $\mu_* \Phi \in \VSO F$ given by $(\mu_* \Phi)(u) \coleq \mu_* ( \Phi ( \mu^* u))$ for $u \in \cD'(N, F)$.

Let $\GSa \in \Gb E\Delta$ and $\mu = \{\mu_G\colon G \to G'\}_{G \in \Delta \cup \{E\}}$ a family of vector bundle isomorphisms over the same diffeomorphism $\underline{\mu}\colon M \to N$. Then $\mu_* \GSa$ is defined as the element of $\Gb {E'}\Delta$ with $\Delta' \coleq \{ G' \}_{G \in \Delta}$ given by
\[ (\mu_* \GSa) ( \Phi' ) \coleq (\mu_E)_* ( \GSa ( \mu^*\Phi' ) ) \]
where $\mu^* \Phi' \coleq (\mu^*_G \Phi'_{G'})_{G} \in \VSO \Delta$ for $\Phi' \in \VSO {\Delta'}$.
\end{definition}
We have $(\mu \circ \nu)_* = \mu_* \circ \nu_*$ and $\id_* = \id$ both for vector smoothing operators and for generalized sections.

\begin{proposition}The action of vector bundle isomorphisms of Definition \ref{def_vbiso} extends the classical one on distributional sections via $\iota$, i.e., $\mu_* \circ \iota = \iota \circ \mu_*$, and trivially the one of smooth sections.
\end{proposition}
\begin{proof}
 Let $u \in \cD'(M,E)$ and $\mu\colon E \to F$ a vector bundle isomorphism. Then for $\Phi \in \VSO F$, $(\mu_*(\iota u))(\Phi) = \mu_* ((\iota u)(\mu^*\Phi)) = \mu_*((\mu^*\Phi)(u)) = \mu_* ( \mu^*(\Phi(\mu_* u))) = \iota(\mu_* u)(\Phi)$.
\end{proof}

\subsection{Lie derivatives}

We assume all vector bundles to be natural from now on, which means that each $E \to M$ is given by $E = \cF(M)$ for a vector bundle functor $\cF$, as for such vector bundles the Lie derivative of a section along a vector field $X$ exists (\cite[6.15]{KMS}). A vector bundle functor $\cF$ is a functor which assigns to each manifold $M$ of fixed dimension $n$ a vector bundle $\cF(M) \to M$ and to each smooth map $\SMa \colon M \to N$, where $N$ is another manifold, a vector bundle homomorphism $\cF(\SMa) \colon \cF(M) \to \cF(N)$ over $\SMa$ which is a linear isomorphism on each fiber. The Lie derivative of $s \in \Gamma(\cF(M))$ with respect to $X \in \fX(M)$ then is defined as $\Lie_X s \coleq \frac{\ud}{\ud t}|_{t=0} (\Fl^X_t)^*s$, where $(\Fl^X_t)^*s \coleq \cF(\Fl^X_{-t}) \circ s \circ \Fl^X_t$.

The usual way to define the Lie derivative of Colombeau generalized functions is by differentiating the pullback along the flow of a complete vector field $X$ with respect to the time parameter (see e.g.~\cite[Definition 3.8]{global}). For this we first calculate the Lie derivative of $\Phi \in \VSO E$ as 
\[ \left.\frac{\ud}{\ud t}\right|_{t=0} ( \Fl^X_t)^* \Phi = \Lie_X \circ \Phi - \Phi \circ \Lie_X \]
and define $\Lso_X \Phi \coleq \Lie_X \circ \Phi - \Phi \circ \Lie_X$ for all $X \in \fX(M)$. Note that $\Lso_X \in \Linc(\VSO E, \VSO E)$. For $\Phi = (\Phi_G)_G \in \VSO \Delta$ we will write $\Lso_X \Phi$ in place of $(\Lso_X \Phi_G)_G$.
Accordingly, the Lie derivative of $\GSa \in \Gb E\Delta$ for complete $X$ is calculated as
\begin{align}
 \left( \left. \frac{\ud}{\ud t}\right|_{t=0} (\Fl^X_t)^* \GSa\right)(\Phi) &= \left.\frac{\ud}{\ud t}\right|_{t=0} (( \Fl^X_t)^*\GSa)(\Phi) \nonumber \\
&= \left.\frac{\ud}{\ud t}\right|_{t=0} (\Fl^X_t)^*(\GSa((\Fl^X_{-t})^*\Phi)) \nonumber \\
&= \left.\frac{\ud}{\ud t}\right|_{t=0} (\Fl^X_t)^* ( \GSa(\Phi)) + \left.\frac{\ud}{\ud t}\right|_{t=0} \GSa((\Fl^X_{-t})^*\Phi) \nonumber \\
&= \Lie_X(\GSa(\Phi)) - \ud \GSa(\Phi)(\Lso_X \Phi)\label{lieder}.
\end{align}
We adopt formula \eqref{lieder} for the definition of the Lie derivative $\Lieh_X$ on $\Gb E\Delta$ for arbitrary $X \in \fX(M)$. However, this is not the only way to define a Lie derivative: similarly to the definitions of the elementary operations on the basic space one can define a Lie derivative $\Liet_X$ `after regularization', i.e., for fixed $\Phi$. Incidentally, this is exactly the one which is used in special Colombeau algebras (cf.~\cite[Definition 3.2.2]{GKOS}). Hence, we have the following two Lie derivatives on $\Gb E\Delta$:
\begin{definition}\label{deflieder}Let $\GSa \in \Gb E\Delta$, $X \in \fX(M)$ and $\Phi \in \VSO \Delta$. Then we define $\Lieh_X \GSa$ and $\Liet_X \GSa$ as the elements of $\Gb E\Delta$ given by
\begin{align*}
(\Lieh_X \GSa)(\Phi) & \coleq \Lie_X ( \GSa (\Phi)) - \ud \GSa(\Phi)(\Lso_X \Phi), \\
(\Liet_X \GSa)(\Phi) & \coleq \Lie_X ( \GSa (\Phi)).
\end{align*}
\end{definition}
These derivatives will be shown to agree on embedded distributions on the level of association in Section \ref{sec_association}. Because for fixed $\GSa \in \Gb E \Delta$ the map $X \mapsto \Liet_X \GSa$ is in $\Hom_{C^\infty(M)} ( \fX(M), \Gb E\Delta) \cong \Hom_{\Gb M \Delta}(\Gb{TM}\Delta, \Gb E \Delta)$, $\Liet_X$ has a natural extension to the case where $X$ is a generalized vector field as follows:

\begin{definition}
For $\GSa \in \Gb E\Delta$ and $X \in \Gb{TM}\Delta$ we define $\Liet_X \GSa \in \Gb E\Delta$ by 
\[ (\Liet_X \GSa)(\Phi) \coleq \Liet_{X(\Phi)} \GSa(\Phi) \qquad (\Phi \in \VSO \Delta). \]
\end{definition}

\begin{proposition}\label{prop_derivprop}The Lie derivatives $\Lieh_X$ and $\Liet_X$ have the following properties:
\begin{enumerate}[label={(\roman*)}]
\item \label{prop_derivprop.1}Let $E \to M$ and $F \to M$ be vector bundles, $\GSa \in \Gb E\Delta$ and $\GSb \in \Gb F\Delta$. Then
\begin{align*}
 \Lieh_X (\GSa \otimes \GSb) &= \Lieh_X \GSa \otimes \GSb + \GSa \otimes \Lieh_X \GSb \qquad (X \in \fX(M)), \\
 \Liet_X (\GSa \otimes \GSb) &= \Liet_X \GSa \otimes \GSb + \GSa \otimes \Liet_X \GSb \qquad (X \in \Gb {TM}\Delta).
\end{align*}
\item \label{prop_derivprop.2}$\Lieh$ and $\Liet$ are $\bK$-bilinear as maps $\fX(M) \times \Gb E\Delta \to \Gb E\Delta$. For $\GFa \in \Gb M\Delta$, $\Liet_X \GFa$ is $\Gb M\Delta$-linear in $X$.
\item\label{5.3} $\Lieh_X$ commutes with $\iota$, i.e., $\Lieh_X \circ \iota = \iota \circ \Lie_X$.
\item\label{5.4} On $\Gamma(E) = \Gb E\emptyset$, $\Lieh_X$ and $\Liet_X$ coincide with the classical Lie derivative $\Lie_X$ of smooth sections.
\end{enumerate}
\end{proposition}
\begin{proof}
\ref{prop_derivprop.1}: For $\Phi \in \VSO \Delta$ we trivially have
\[ \Lie_X (( \GSa \otimes \GSb) ( \Phi)) = \Lie_X ( \GSa(\Phi) ) \otimes \GSb + \GSa(\Phi) \otimes \Lie_X (\GSb(\Phi)) \]
and, by the chain rule,
\[ \ud ( \GSa \otimes \GSb )(\Phi)(\Lso_X \Phi) = (\ud \GSa)(\Phi)(\Lso_X \Phi) \otimes \GSb(\Phi) + \GSa(\Phi) \otimes (\ud \GSb)(\Phi)(\Lso_X \Phi) \]
which gives the claim. 
\ref{prop_derivprop.2} and \ref{5.4} are clear, while \ref{5.3} is seen from $\Lieh_X ( \iota u )(\Phi) = - (\Lso_X \Phi)(u) + \Lie_X ( \Phi(u)) = \Phi ( \Lie_X u) = \iota ( \Lie_X u)(\Phi)$ for $u \in \cD'(M, E)$ and $\Phi \in \VSO E$.
\end{proof}

The Lie derivatives $\Liet_X$ and $\Lieh_X$ extend to the tensor algebra $\Gtb E\Delta$.

Classically, the Lie bracket $[\Lie_X, \Lie_Y]$ of two derivations on $C^\infty(M)$ induced by smooth vector fields $X,Y \in \fX(M)$ is again induced by a smooth vector field denoted by $[X,Y]$ and given by $[X,Y] = \Lie_X Y$. The same holds for $\Liet$ if we take generalized vector fields $X,Y \in \Gb {TM} \Delta$ and define their Lie bracket $[X,Y]$ as follows:
\begin{definition}Let $X,Y \in \Gb{TM}\Delta$. Then their Lie bracket $[X,Y] \in \Gb {TM}\Delta$ is defined as
\[ [X,Y] (\Phi) \coleq [X(\Phi), Y(\Phi)] \qquad (\Phi \in \VSO \Delta). \]
\end{definition}
Then $\Liet_X \circ \Liet_Y - \Liet_Y \circ \Liet_X = \Liet_{[X,Y]}$ and $[X,Y] = \Liet_X Y$ for $X,Y \in \Gb {TM}\Delta$. Moreover, $[X,Y]$ has the same properties as the classical Lie bracket: it is $\bK$-bilinear, antisymmetric, satisfies the Jacobi identity and for $\GFa,\GFb \in \Gb M \Delta$, $[\GFa X, \GFb Y] = \GFa \GFb  [X,Y] + (\GFa \Liet_X \GFb) Y - (\GFb \Liet_X \GFa)X$.

One may ask why it is necessary to consider two notions of Lie derivatives. Although both are natural in their own right, it results from the Schwartz impossibility result that one cannot have a Lie derivative $\Lie_X$ of generalized functions which at the same time commutes with the embedding of distributions and is $C^\infty$-linear in the direction $X$. The first property is indispensable for a geometric theory, and the second one is needed in a way for defining quantities like the curvature tensor. Hence, it can be understood that the notion of Lie derivative splits into two, each having one of the properties mentioned, but satisfying the other one in the sense of association, as we will see.

\section{Test objects}\label{sec_testobjects}

\def\STO{{\notationcolor A}} 
\def\STOz{{\notationcolor A}} 

The task of finding the right class of test objects for the quotient construction can without doubt be considered the decisive step in the development of diffeomorphism invariant Colombeau algebras. Needless to say, much variance is possible in this choice, which in turn is directly reflected in the properties of the algebras so obtained. It is therefore desirable to start with a very general class of test objects on which further properties can then be imposed as required for specific applications. We will describe such a general class for the construction of generalized section spaces in this section.

\begin{definition}A \emph{test object} on a vector bundle $E \to M$ is a net $(\Phi_\e)_\e \in \VSO E^I$ satisfying the following conditions.
\begin{enumerate}
 \item[(VSO1)] For any Riemannian metric $g$ on $M$ $\forall x_0 \in M$ $\exists$ an open neighborhood $V$ of $x_0$ $\forall r>0$ $\exists \e_0 \in I$ $\forall x \in V$ $\forall \e \le \e_0$ $\forall u \in \cD'(M, E)$: $(u|_{B^g_{r}(x)} = 0 \Rightarrow \Phi_\e(u)(x) = 0)$.
\item[(VSO2)] $\Phi_\e \to \id$ in $\Lincb(\cD'(M, E), \cD'(M, E))$.
\item[(VSO3)] $\forall \csna \in \csn{\VSO E}$ $\exists N \in \bN$: $\csna(\Phi_\e) = O(\e^{-N})$.
 \item[(VSO4)] $\forall \csna \in \csn{\Lincb(\Gamma(E), \Gamma(E))}$ $\forall m \in \bN$: $\csna(\Phi_\e|_{\Gamma(E)} - \id) = O(\e^m)$.
\end{enumerate}
A \emph{0-test object} is a sequence $(\Phi_\e)_\e \in \VSO{E}^I$ satisfying (VSO1), (VSO3) and the following conditions.
\begin{enumerate}
\item[(VSO2')] $\Phi_\e \to 0$ in $\Lincb(\cD'(M, E), \cD'(M, E))$.
\item[(VSO4')] $\forall \csna \in \csn { \Lincb(\Gamma(E), \Gamma(E))}$ $\forall m \in \bN$: $\csna(\Phi_\e|_{\Gamma(E)}) = O(\e^m)$.
\end{enumerate}
In other words, $(\Phi_\e)_\e$ is a $0$-test object if and only if $(\Phi_\e + \id)_\e$ is a test object. We denote by $\TOv E$ the set of all test objects, by $\TOvz E$ the set of all $0$-test objects, and by $\TOvl E$ the set of all $(\Phi_\e)_\e \in \VSO E^I$ satisfying (VSO1).
$\TOvz E$ and $\TOvl E$ are are $C^\infty(M)$-modules with multiplication $f \cdot (\Phi_\e)_\e \coleq (m_f \circ \Phi_\e)_\e$, as is easily verified. $\TOv E$ is an affine space over $\TOvz E$; in fact, for finitely many $f_i \in C^\infty(M)$ with $\sum_i f_i = 1$ and $(\Phi_{i,\e})_\e \in \TOv E$, $(\sum_i f_i \Phi_{i,\e})_\e \in \TOv E$ again.

Furthermore, we write $\TOv{U,E} \coleq \TOv{E|_U}$, $\TOvz{U,E} \coleq \TOvz{E|_U}$ as well as $\TOvl{U,E} \coleq \TOvl{E|_U}$. The notations $\TOv\Delta$, $\TOvz\Delta$ and $\TOvl\Delta$ mean families of corresponding test objects indexed by $\Delta$, e.g.,
\[ \TOv\Delta \coleq \prod_{G \in \Delta} \TOv G = \{ ( ( \Phi_{G,\e} )_\e )_G : (\Phi_{G,\e})_\e \in \TOv G\ \forall G \in \Delta\ \}. \]
\end{definition}

\begin{remark}\label{remarkiblah}
 \begin{enumerate}[label=(\roman*)]
  \item Note that (VSO1) is independent of the Riemannian metric because any two given Riemannian metrics are equivalent locally (see e.g.~\cite[(3.68)]{GKOS}x or \cite[Lemma 2]{rm_diss}).
  
  \item We point out that we would not strictly need (VSO1) for the construction of a Colombeau generalized function space, but because it is essential for obtaining the sheaf property we include it in the definition of test objects in any case.
  Moreover, if (VSO1) holds for certain $(V, r)$ it obviously also holds for all open subsets of $V$ and $r'>r$, and for all relatively compact sets $V$.

  \item \label{remarkiblah.3} It is instructive to compare the above test objects with those which have been employed for the construction of $\hg M$ (\cite[Definition 3.4]{global2}). Restricting our considerations to the scalar case, (VSO1) is more easily understood in terms of smoothing kernels instead of smoothing operators, as for $\vec\varphi_\e$ corresponding to $\Phi_\e$ it simply states that $\supp \vec\varphi_\e(x)$ is eventually contained in an arbitrarily small neighborhood of $x$ in a locally uniform but otherwise arbitrary way. In contrast, \cite{global2} requires that the support of $\vec\varphi_\e(x)$ shrinks linearly with $\e$. Furthermore, instead of (VSO3) and (VSO4), in \cite{global2} one only has estimates on derivatives of $\vec\varphi_\e$ of the form $\pd_y^\alpha \pd_{x+y}^\beta \vec\varphi(x)(y)$ (in local coordinates) and demands convergence to $\id$ merely in the topology of uniform convergence on compact sets instead of the topology of uniform convergence on compact sets in all derivatives. Because one cannot avoid having to use properties (VSO2-4) in the construction of Colombeau algebras they are in effect also used in \cite{global2} but are proved from the other properties of test objects there. This is less natural and efficient than the definitions we start with above. Furthermore, because of these more restrictive properties, obtaining diffeomorphism invariance is much more involved than in our case, where it essentially follows from Lemma \ref{diffeoinv} \ref{diffeoinv.2}.
 \end{enumerate}
\end{remark}

We will also need sets of test objects for which the above conditions hold uniformly:
\begin{definition}\label{def_uniform}
 A set $\STO \subset \TOv{E}$ of test objects is called \emph{uniform} if conditions (VSO1)--(VSO4) hold uniformly for all of its elements, i.e.:
\begin{enumerate}[label={(UVSO\arabic*)}]
 \item[(UVSO1)] For any Riemannian metric $g$ on $M$ $\forall x_0 \in M$ $\exists$ an open neighborhood $V$ of $x_0$ $\forall r>0$ $\exists \e_0 \in I$ $\forall (\Phi_\e)_\e \in \STO$ $\forall x \in V$ $\forall \e \le \e_0$ $\forall u \in \cD'(M, E)$: $(u|_{B^g_{r}(x)} = 0 \Rightarrow \Phi_\e(u)(x) = 0)$. 
\item[(UVSO2)] $\forall \csna \in \csn {\Lincb(\cD'(M,E), \cD'(M,E))}$: $\sup_{(\Phi_\e)_\e \in \STO} \csna(\Phi_\e - \id) \to 0$.
\item[(UVSO3)] $\forall \csna \in \csn{\VSO E}$ $\exists N \in \bN$: $\sup_{(\Phi_\e)_\e \in \STO} \csna(\Phi_\e) = O(\e^{-N})$.
 \item[(UVSO4)] $\forall \csna \in \csn{\Lincb(\Gamma(E), \Gamma(E))}$ $\forall m \in \bN$: $\sup_{(\Phi_\e)_\e \in \STO}\csna(\Phi_\e|_{\Gamma(E)} - \id) = O(\e^m)$.
\end{enumerate}
Similarly, a set $\STOz \subset \TOvz{E}$ of 0-test objects is called \emph{uniform} if it satisfies (UVSO1), (UVSO3) and the following conditions:
\begin{enumerate}
\item[(UVSO2')] $\forall \csna \in \csn{\Lincb(\cD'(M, E), \cD'(M, E))}$: $\sup_{(\Phi_\e)_\e \in \STOz} \csna(\Phi_\e) \to 0$.
\item[(UVSO4')] $\forall \csna \in \csn { \Lincb(\Gamma(E), \Gamma(E))}$ $\forall m \in \bN$: $\sup_{(\Phi_\e)_\e \in \STOz} \csna(\Phi_\e|_{\Gamma(E)}) = O(\e^m)$.
\end{enumerate}
\end{definition}
Subsets of $\TOv\Delta$ or $\TOvz\Delta$ are called uniform if each of their components in $\TOv G$ or $\TOvz G$ for $G \in \Delta$ is uniform. Such uniform sets of (0-) test objects will be employed in Theorem \ref{nonegder}, where we show that negligibility of moderate generalized functions can be tested without resorting to derivatives. The following Lemma shows why the concept of 0-test objects is needed, namely in order to have $\Lieh_X$ preserve moderateness and negligibility.

\begin{lemma}\label{diffeoinv}
\begin{enumerate}[label=(\roman*)]
 \item $\Lso_X$ acting on $(\Phi_\e)_\e \in \VSO E^I$ transforms the above properties in the following way: (VSO1) $\Rightarrow$ (VSO1), (VSO2) $\Rightarrow$ (VSO2'), (VSO3) $\Rightarrow$ (VSO3), (VSO4) $\Rightarrow$ (VSO4'), and similarly for the uniform conditions. Hence, it maps $\TOv E$ into $\TOvz E$.
 \item \label{diffeoinv.2} Vector bundle isomorphisms as in Definition \ref{def_vbiso} preserve all the above properties (VSO1)--(VSO4), (VSO2'), (VSO4'), and their uniform variants.
\end{enumerate}
\end{lemma}
\begin{proof} \checkedproof
(i) Fix a Riemannian metric $g$ and $x_0 \in M$. Because $(\Phi_\e)_\e$ satisfies (VSO1) there exists an open neighborhood $V$ of $x_0$ such that $\forall r > 0$ $\exists \e_0$ $\forall x \in V$ $\forall \e \le \e_0$ $\forall u \in \cD'(M,E)$:
$( u|_{B^g_{r/2}(x)} = 0 \Rightarrow \Phi_\e(u)(x) = 0 )$.
Now suppose $u$ vanishes on $B^g_r(x)$ for some $x \in V$. Choose an open neigborhood $V' \subseteq V \cap B^g_{r/2}(x)$ of $x$. For any $y \in V'$, $u|_{B^g_{r/2}(y)}=0$ holds and implies $\Phi_\e(u)(y)=0$, which gives $\Lie_X(\Phi_\e(u))(x)=0$. Moreover, $(\Lie_X u)|_{B^g_{r/2}(x)}=0$ implies $\Phi_\e(\Lie_X u)(x)=0$. In sum, $(\Lso_X \Phi_\e)(u)(x) = \Lie_X ( \Phi_\e(u))(x) - \Phi_\e ( \Lie_X u)(x) = 0$, which means that $(\Lso_X \Phi_\e)_\e$ satisfies (VSO1).

For the other conditions note that $\Lso_X = (\Lie_X)_* - (\Lie_X)^*$, $\Phi \mapsto \Lie_X \circ \Phi - \Phi \circ \Lie_X$ is continuous, hence for $\csna$ a continous seminorm of any of $\Lincb(\Gamma(E), \Gamma(E))$, $\Lincb(\VSO E, \VSO E)$ or $\Lincb(\cD'(M, E), \cD'(M,E))$, $\csna \circ \Lso_X$ is also a continuous seminorm of the same space, which implies the claim. Uniformity simply goes through.

(ii) The claim for (VSO1) follows by taking the pullback metric.
The rest follows because $\csna \mapsto \csna \circ \mu_*$ transforms all involved seminorms appropriately.
\end{proof}

\begin{remark}\label{remarki}
\begin{enumerate}[label=(\roman*)]
 \item The above conditions are the bare minimum needed for the construction of a Colombeau algebra. In practice, many more conditions can be added easily e.g.~for stronger association properties.
 \item \label{remarki.2} Because $\Gamma(E)$ and $\cD'(M,E)$ are Montel spaces, the Banach-Steinhaus theorem implies that conditions (VSO2,2',3,4,4') are already satisfied if they hold for the weak topologies on the respective spaces of linear mappings. \proofs{\#42}
 \end{enumerate}
 \end{remark}

Smoothing operators and test objects can also be restricted in a certain sense, which will be the essential ingredient for obtaining the sheaf property of the Colombeau quotient.

\begin{theorem}\label{sk_restriction}
For any open subsets $U,V \subseteq M$ with $V \subseteq U$ there exists a linear continuous mapping
$\rSO_{V,U}\colon \VSO{U,E}\to \VSO{V,E}$ such that the following holds:
\begin{enumerate}[label=(\roman*)]
 \item\label{36.1} Given $(\Phi_\e)_\e \in \TOvl{U, E}$ and any Riemannian metric $g$ on $U$, each $x \in V$ has an open neighborhood $X \subseteq V$ such that for some $r_0>0$, $\overline{B^g_{r_0}(X)} \csub V$ and for all $0 < r \le r_0$ there is $\e_0 \in I$ such that
\begin{multline}\label{kalt}
\forall \e \le \e_0\ \forall y \in X\ \forall v \in \cD'(V, E)\ \forall \tilde v \in \cD'(U,E):\\
(v|_{B^g_r(y)} = \tilde v|_{B^g_r(y)} \Rightarrow (\rSO_{V,U} \Phi_\e)(v)(y) = \Phi_\e(\tilde v)(y)).
\end{multline}
\item\label{9.1.1} For $\Phi \in \VSO{U,E}$ and $f \in C^\infty(U)$, $\rSO_{V,U} ( f \cdot \Phi) = f|_V \cdot \rSO_{V,U}(\Phi)$.
\item\label{9.2} For $\Phi, \Psi \in \VSO{U,E}$ and $p \in V$, $\ev_p \circ \Phi = \ev_p \circ \Psi$ implies $\ev_p \circ \rSO_{V,U}\Phi = \ev_p \circ \rSO_{V,U}\Psi$.
\end{enumerate}
\end{theorem}

\begin{proof}\checkedproof
Cover $V$ by a family of open relatively compact subsets $W$ such that $\overline{W} \subseteq V$. Choose a partition of unity $(\chi_W)_W$ on $V$ subordinate to this family and a function $\theta_W \in C^\infty(U)$ for each $W$ such that $\theta_W \equiv 1$ on an open neighborhood of $\supp \chi_W$ and $\supp \theta_W \subseteq V$. Define $\rSO_{V,U} \colon \VSO{U,E} \to \VSO{V,E}$ by
\[ (\rSO_{V,U}\Phi)(u) \coleq \sum_W \chi_W \cdot \Phi(u \cdot \theta_W)|_V\qquad (u \in \cD'(V,E)). \]
Note that $u \cdot \theta_W \in \cD'(U,E)$ here.
It is easily verified that $\rSO_{V,U}\Phi \in \VSO{V,E}$. Any $x \in V$ has an open neighborhood $X \subseteq V$ intersecting only finitely many $\supp \chi_W$, say those for $W=W_1, \dotsc, W_n$. $\overline{X} \subseteq \bigcup_{i=1}^n \overline{\supp \chi_{W_i}} \subseteq \bigcup_{i=1}^n \overline{W_i}$ hence is compact and contained in $V$, and
\begin{equation}\label{blah}
(\rSO_{V,U} \Phi)(u)|_X = \sum_{i=1}^n \chi_{W_i}|_X \cdot \Phi(u \cdot \theta_W)|_X.
\end{equation}
Because $\Lincb(\cD'(V,E), \Gamma(V, E))$ carries the projective topology with respect to all maps $\Phi \mapsto |_X \circ \Phi$ into $\Lincb(\cD'(V,E), \Gamma(X,E))$, where $X$ runs through any open cover of $V$, and $|_X \circ \rSO_{V,U}$ as given by \eqref{blah} is a sum of linear continuous maps, $\rSO_{V,U}$ is continuous. \ref{9.1.1} and \ref{9.2} are clear from the definition.

For \ref{36.1}, fix $(\Phi_\e)_\e \in \TOvl{U, E}$ and a Riemannian metric $g$ on $U$. Take $r_0>0$ so small that $\overline{B^g_{r_0}(X \cap \supp \chi_{W_i})} \subseteq \theta_{W_i}^{-1}(1)$ for all $i$ and hence $\overline{B^g_{r_0}(X)} \csub V$.

Now suppose $0 < r \le r_0$. By (VSO1) and because $X$ is relatively compact there is $\e_0 \in I$ such that $\forall y \in X$ $\forall \e \le \e_0$ $\forall u \in \cD'(U,E)$: ($u|_{B^g_r(y)} = 0 \Rightarrow \Phi_\e(u)(y) = 0$). Let $\e \le \e_0$, $y \in X$, $v \in \cD'(V,E)$ and $\tilde v \in \cD'(U,E)$ be given with $v|_{B^g_r(y)} = \tilde v|_{B^g_r(y)}$.  Now for $y \in X \cap \supp \chi_{W_i}$ we have $(v \cdot \theta_{W_i} - \tilde v)|_{B^g_r(y)} = 0$ and hence
\[
 (\rSO_{V,U} \Phi_\e)(v)(y) - \Phi_\e(\tilde v)(y) = \sum_{i=1}^n \chi_{W_i}(y) \cdot \Phi_\e(v \cdot \theta_{W_i} - \tilde v)(y) = 0.\qedhere
\]
\end{proof}

\begin{remark}\label{alpharemark}In practice, \eqref{kalt} gives a way to evaluate $\rSO_{V,U} \Phi_\e$: given a relatively compact subset $X$ with $\overline{X} \subseteq V$, take $f \in C^\infty(U)$ with $\supp f \subseteq V$ and $f \equiv 1$ on an open neighborhood of $\overline{X}$. Then for small $\e$ we will have $(\rSO_{V,U} \Phi_\e)(v)|_X = \Phi_\e(f \cdot v)|_X$. Similarly, one sees that for $u \in \cD'(U, E)$, $\Phi_\e(u)|_X = (\rSO_{V,U} \Phi_\e )(u|_V)|_X$ for small $\e$.
\end{remark}

The mapping of Theorem \ref{sk_restriction} turns $U \mapsto \TOvl{U,E}$ into a presheaf if we consider those smoothing operators which agree \emph{locally} and \emph{eventually} in the following sense to be equivalent.

\begin{definition}
Let $\TOvn{U,E}$ denote the $C^\infty(U)$-module of all nets $(\Phi_\e)_\e \in \VSO{U,E}^I$ such that for each $x \in U$ there is an open neighborhood $V$ in $U$ and $\e_0 \in I$ such that for $\e \le \e_0$, $|_V \circ \Phi_\e = 0$. Then $\TOvn{U,E}$ is a submodule of $\TOvz{U,E} \subseteq \TOvl{U,E}$ and we define the quotient modules $\TOqvz{U,E} \coleq \TOvz{U,E} / \TOvn{U,E}$ and $\TOqvl{U, E} \coleq \TOvl{U,E} / \TOvn{U,E}$. For $(\Phi_\e)_\e, (\Psi_\e)_\e \in \TOvl{U,E}$ we write $(\Phi_\e)_\e \sim (\Psi_\e)_\e$ if $(\Phi_\e)_\e - (\Psi_\e)_\e \in \TOvn{U,E}$.

We define an equivalence relation on $\TOv{U,E}$ by setting $(\Phi_\e)_\e \sim (\Psi_\e)_\e$ if $(\Phi_\e)_\e - (\Psi_\e)_\e \in \TOvn{U,E}$, where $(\Phi_\e)_\e$,$(\Psi_\e)_\e$ in $\TOv{U, E}$. By $\TOqv{U,E} \coleq \TOv{U,E} / \sim$ we denote the quotient set, which is an affine space over $\TOqvz{U,E}$.
\end{definition}

We will now show that $U \mapsto \TOqvl{U,E}$ is a sheaf of $C^\infty$-modules. The first step is the presheaf property.

\begin{proposition}Let $(\Phi_\e)_\e \in \TOvl{U,E}$. Then
\begin{enumerate}[label=(\roman*)]
\item\label{9.0} $(\rSO_{V,U} \Phi_\e)_\e \in \TOvl{V,E}$,
\item\label{9.2.1} $(\Phi_\e)_\e \sim 0$ implies $(\rSO_{V,U} \Phi_\e)_\e \sim 0$,
\item\label{9.3} for open sets $U_2 \subseteq U_1 \subseteq U$ we have $((\rSO_{U_2, U_1} \circ \rSO_{U_1, U})( \Phi_\e))_\e \sim (\rSO_{U_2, U} \Phi_\e)_\e$.
\end{enumerate}
\end{proposition}
\begin{proof}

\ref{9.0}: 
Fix $x_0 \in V$ and a Riemannian metric $g$ on $U$. Choose $X$ and $r_0$ as in Theorem \ref{sk_restriction} \ref{36.1}. Fix some $0 < r \le r_0$ for showing (VSO1). Choose $f \in C^\infty(U)$ with $\supp f \subseteq V$ and $f \equiv 1$ on $B^g_{r_0}(X)$. Insert $r$ in Theorem \ref{sk_restriction} \ref{36.1} and (VSO1) (with $V$ there given by the relatively compact set $X$), which gives $\e_0$ and $\e_1$; let $\e \le \min(\e_0, \e_1)$, $y \in X$ and $v \in \cD'(V,E)$ with $v|_{B^g_r(y)} = 0$. Then $f \cdot v \in \cD'(U,E)$ satisfies $(f \cdot v)|_{B^g_r(y)} = v|_{B^g_r(y)} = 0$, hence $(\rSO_{V,U} \Phi_\e)(v)(y) = \Phi_\e(f \cdot v)(y) = 0$.

For \ref{9.2.1}, let $x \in V$. There is an open neighborhood $W$ of $x$, which can be taken relatively compact such that $\overline{W} \csub V$, such that $|_W \circ \Phi_\e = 0$ for small $\e$. With $f \in C^\infty(U)$ such that $\supp f \subseteq V$ and $f \equiv 1$ on an open neighborhood of $\overline{W}$ we have $(\rSO_{V,U}\Phi_\e)(u)|_W = \Phi_\e(f \cdot u)|_W = 0$ for all $u \in \cD'(V,E)$ and small $\e$ by assumption.

For \ref{9.3} we only have to note that for all open relatively compact sets $X$ with $\overline{X} \subseteq U_2$ and $f \in C^\infty(U)$ with $f \equiv 1$ on an open neighborhood of $\overline{X}$ and $\supp f \subseteq U_2$, for small $\e$ we have
\[
(\rSO_{U_2, U_1} ( \rSO_{U_1, U} \Phi_\e )) (u)|_X = \rSO_{U_1,U}(\Phi_\e)(f \cdot u)|_X = \Phi_\e(f \cdot u)|_X  = \rSO_{U_2, U}(\Phi_\e)(u)|_X.\qedhere
\]
\end{proof}
In conjunction with Theorem \ref{sk_restriction} \ref{9.1.1} this means that $U \mapsto \TOqvl {U,E}$ is a presheaf of $C^\infty$-modules, where the restriction map $|_V\colon \TOqvl{U,E} \to \TOqvl{V,E}$ is determined by $(\Phi_\e)_\e \mapsto (\rSO_{V,U}\Phi_\e)_\e$. We will now show that we can also uniquely glue together coherent families:

\begin{proposition}\label{sk_sheaf}
 $U \mapsto \TOqvl{U,E}$ is a sheaf of $C^\infty$-modules on $M$.
\end{proposition}
\begin{proof}
Let $U \subseteq M$ be open, $(U_\lambda)_\lambda$ an open cover of $U$ and $[(\Phi_\e)_\e] \in \TOqvl{U,E}$. Supposing that $[ (\Phi_\e)_\e ]|_{U_\lambda} = 0$ for each $\lambda$, we have to show that $[(\Phi_\e)_\e]=0$. Let $x \in U$. Then there is an open neighborhood $W$ of $x$ such that $\overline{W}$ is compact and contained in $U_\lambda$ for some $\lambda$. Then $|_W \circ \Phi_\e = |_W \circ \rSO_{U_\lambda, U} \Phi_\e \circ |_{U_\lambda} = 0$ by small $\e$ by assumption, which gives the claim.

Now let $[(\Phi_\e^\lambda)_\e] \in \TOqvl{U_\lambda,E}$ be given for each $\lambda$, satisfying $[ (\Phi_\e^\lambda)_\e]|_{U_\lambda \cap U_\mu} = [ (\Phi_\e^\mu)_\e]|_{U_\lambda \cap U_\mu}$ $\forall \lambda,\mu$. Let $(\chi_\lambda)_\lambda$ be a partition of unity on $U$ subordinate to $(U_\lambda)_\lambda$.
For each $\e \in I$ we define $\Phi_\e \in \VSO{U,E}$ by
\begin{equation}\label{def_skglue}
 \Phi_\e(u) \coleq \sum_\lambda \chi_\lambda \cdot \Phi^{\lambda}_\e(u|_{U_\lambda}) \qquad (u \in \cD'(U,E)).
\end{equation}
We claim that $(\Phi_\e)_\e$ satisfies (VSO1): fix $x_0 \in U$, a Riemannian metric $g$ on $U$ and an open, relatively compact neighborhood $V$ of $x_0$ intersecting only finitely many $\supp \chi_\lambda$, namely those for $\lambda$ in some finite index set $F$. Given $r>0$, for each $\lambda \in F$ choose $\e_\lambda$ according to (VSO1) such that $\forall x \in \supp \chi_\lambda \cap V$ $\forall \e \le \e_\lambda$ $\forall u \in \cD'(U_\lambda, E)$: ($u|_{B^{g_\lambda}_r(x)}=0 \Rightarrow \Phi^\lambda_\e(u)(x) = 0$). Here, $B^{g_\lambda}_r(x)$ denotes the metric ball at $x$ of radius $r$ in $U_\lambda$ with respect to the restricted metric $g_\lambda \coleq g|_{U_\lambda}$. Then for $\e  \le \min_{\lambda \in F} \e_\lambda$, $x \in V$ and $u \in \cD'(U, E)$ with $u|_{B^g_r(x)} = 0$, for $x \in \supp \chi_\lambda \cap V$ we also have $(u|_{U_\lambda})|_{B^{g_\lambda}_r}(x) = 0$ and hence $\Phi_\e(u)(x) = \sum_{\lambda \in F} \chi_\lambda(x) \Phi^\lambda_\e (u|_{U_\lambda})(x) = 0$.

Now we are going to show that $[(\Phi_\e)_\e]|_{U_\mu} = [(\Phi_\e^\mu)_\e]$. For this we need that $\forall x_0 \in U_\mu$ there is an open neighborhood $V$ of $x_0$ in $U_\mu$ and $\e_0 \in I$ such that
$(\rSO_{U_\mu, U} \Phi_\e)(u)|_V = \Phi^\mu_\e(u)|_V$ $\forall u \in \cD'(U_\mu, E)$ for small $\e$.

First, choose $V \subseteq U$ open such that $\overline{V}$ is compact and contained in $U_\mu$.
Fix $f \in C^\infty(U)$ with $f \equiv 1$ on an open neighborhood of $\overline{V}$ and $\supp f \subseteq U_\mu$. Then for small $\e$, the following identities show the claim:
\begin{align*}
(\rSO_{U_\mu, U}\Phi_\e)(u)|_V & = \Phi_\e(f \cdot u)|_V
= \sum_{\lambda \in F}  \chi_\lambda|_{V} \cdot \Phi^\lambda_\e ( (f \cdot u)|_{U_\lambda})|_{\carr \chi_\lambda \cap V} \label{onestar} \\
& = \sum_{\lambda \in F} \chi_\lambda|_{V} \cdot (\rSO_{U_\lambda \cap U_\mu, U_\lambda} \Phi_\e^\lambda)(u|_{U_\lambda \cap U_\mu})|_{\carr \chi_\lambda \cap V} \\
& = \sum_{\lambda \in F} \chi_\lambda|_{V} \cdot (\rSO_{U_\lambda \cap U_\mu, U_\mu} \Phi^\mu_\e)(u|_{U_\lambda \cap U_\mu})|_{\carr \chi_\lambda \cap V} \\
&= \sum_{\lambda \in F} \chi_\lambda|_V \cdot \Phi^\mu_\e (u)|_{\carr \chi_\lambda \cap V} = \Phi^\mu_\e(u)|_V,
\end{align*}
where we have used Remark \ref{alpharemark} because $\overline{\carr \chi_\lambda \cap V}$ is compact and contained in $U_\lambda \cap U_\mu$.
\end{proof}

The following implies that also $U \mapsto \TOqv{U,E}$ and $U \mapsto \TOqvz{U,E}$ are sheaves:

\begin{proposition}\label{skproploc} Let $U \subseteq M$ open and $(U_\lambda)_\lambda$ an open cover of $U$. Then $(\Phi_\e)_\e \in \TOvl{U,E}$ satisfies (VSO2), (VSO2'), (VSO3), (VSO4) or (VSO4') if and only if each $(\rSO_{U_\lambda, U} \Phi_\e)_\e \in \TOvl{U_\lambda, E}$ does so.
\end{proposition}
\begin{proof} \checkedproof \proofs{\#24}
The proof is straightforward; for example, assume that $(\Phi_\e)_\e$ satisfies (VSO2). By Remark \ref{remarki} \ref{remarki.2} it suffices to show weak convergence. Let $u \in \cD'(U_\lambda, E)$ and $\varphi \in \Gamma_c(U_\lambda, E^* \otimes \Vol(M))$. Choose $f \in C^\infty(U)$ such that $f \equiv 1$ on an open neighborhood of $\supp \varphi$ and $\supp f \subseteq U_\lambda$. Then
\[ \langle (\rSO_{U_\lambda,U} \Phi_\e)(u), \varphi) \rangle = \langle \Phi_\e(fu), \varphi \rangle \to \langle fu, \varphi \rangle = \langle u , \varphi \rangle, \]
which means that $(\rSO_{U_\lambda, U} \Phi_\e)_\e$ satisfies (VSO2). For the converse direction a similar calculation using a partition of unity subordinate to $(U_\lambda)_\lambda$ is employed. The other properties are shown analogously.
\end{proof}

\begin{corollary}\label{sk_soft}
Let $W,V,U \subseteq M$ be open such that $\overline{W} \subseteq V \cap U \neq \emptyset$, and $(\Phi_\e)_\e \in \TOv{V,E}$. Then there exists $(\Psi_\e)_\e \in \TOv{U,E}$ such that $(\rSO_{W,U} \Psi_\e)_\e \sim (\rSO_{W,V} \Phi_\e)$. An analogous statements holds for $(\Phi_\e)_\e$ in $\TOvz{V,E}$ or $\TOvl{V,E}$.
\end{corollary}
\begin{proof}
 Choose an open neighborhood $W'$ of $\overline{W}$ such that $\overline{W'} \subseteq U \cap V$, any $(\Psi^0_\e)_\e \in \TOv{U,E}$ and $\chi \in C^\infty(U \cap V)$ with $\supp \chi \subseteq W'$ and $\chi \equiv 1$ on $W$. Then $\chi \cdot [(\Phi_\e)_\e]|_{U \cap V} + (1 - \chi) \cdot [(\Psi^0_\e)_\e]|_{U \cap V}$ is an element of $\TOv{U \cap V, E}$; its restriction to $(U \cap V) \cap (U \setminus \overline{W'})$ is given by $[(\Psi^0_\e)_\e]|_{(U \cap V)\setminus \overline{W'}}$. Hence, by Proposition \ref{sk_sheaf} there exists $(\Psi_\e)_\e \in \TOv{U,E}$ such that $[(\Psi_\e)_\e]|_W = [(\Phi_\e)_\e]|_W$.
\end{proof}

Furthermore, we will need the following Lemma later on.

\begin{lemma}\label{sk_cut}
 Let $K,V,U \subseteq M$ be open, $V \subseteq U$ and $\overline{K}$ compact and contained in $V$. Then if one of $(\Phi_\e)_\e \in \TOv{V,E}$ and $(\Psi_\e)_\e \in \TOv{U,E}$ is given the other one can be chosen such that $|_K \circ \Phi_\e \circ |_V = |_K \circ \Psi_\e$ on $\cD'(U,E)$; similarly for the spaces $\STOvz$ and $\STOvl$.
\end{lemma}
\begin{proof}
 Choose $W$ open with $\overline{K} \subseteq W \subseteq \overline{W} \subseteq V$. Then by Corollary \ref{sk_soft} $(\Psi_\e)_\e$ or $(\Phi_\e)_\e$ exists such that $(\rSO_{W,U}\Psi_\e)_\e \sim (\rSO_{W,V} \Phi_\e)_\e$. Then, for $u \in \cD'(U,E)$, $\Phi_\e(u|_V)|_K = (\rSO_{W,V} \Phi_\e)(u|_W)|_K = (\rSO_{W,U} \Psi_\e)(u|_W)|_K = \Psi_\e(u)|_K$ for small $\e$.
\end{proof}

In order to study the local expression of test objects we list the relevant conditions for nets of smoothing operators $(\Phi_\e)_\e \in \SO \Omega^I$ for $\Omega \subseteq \bR^n$ open (cf.\ Section \ref{sec_revscalar} and \cite{papernew}).
\begin{enumerate}[label=(\arabic*)]
\item[(SO1)] is identical to (VSO1).
\item[(SO2)] $\Phi_\e \to \id \textrm{ in }\Lincb(\cD'(\Omega), \cD'(\Omega))$.
\item[(SO2')] $\Phi_\e \to 0 \textrm{ in }\Lincb(\cD'(\Omega), \cD'(\Omega))$.
\item[(SO3)] $\forall \csna \in \csn{ \SO \Omega }\ \exists N \in \bN: \csna (\Phi_\e) = O(\e^{-N})$.
\item[(SO4)] $\forall \csna \in \csn{\Lincb(C^\infty(\Omega), C^\infty(\Omega))}\ \forall m \in \bN: \csna (\Phi_\e|_{C^\infty(\Omega)} - \id ) = O(\e^m)$,
\item[(SO4')] $\forall \csna \in \csn{\Lincb(C^\infty(\Omega), C^\infty(\Omega))}\ \forall m \in \bN: \csna (\Phi_\e|_{C^\infty(\Omega)} ) = O(\e^m)$,
\end{enumerate}

Suppose at first that $E$ is trivializable over a chart $(U, \varphi)$ of $M$. Then according to Section \ref{subsec_smoothop}, $\Phi$ has components $\Phi_{ij} \in \SO U \cong \SO \Omega$ with $\Omega = \varphi(U)$. Conditions (VSO1--4,2',4') then translate into conditions on these components as follows:
\begin{align*}
\Phi\textrm{ satisfies } \textrm{(VSO1)} &\Longleftrightarrow \textrm{each }\Phi_{ij} \textrm{ satisfies (SO1)} \\
\Phi\textrm{ satisfies } \textrm{(VSO2)} &\Longleftrightarrow \Phi_{ij} \textrm{ satisfies (SO2) if } i = j \textrm{ and (SO2') if } i \ne j \\
\Phi\textrm{ satisfies } \textrm{(VSO2')} &\Longleftrightarrow \textrm{each }\Phi_{ij} \textrm{ satisfies (SO2')} \\
\Phi\textrm{ satisfies } \textrm{(VSO3)} &\Longleftrightarrow \textrm{each }\Phi_{ij} \textrm{ satisfies (SO3)} \\
\Phi\textrm{ satisfies } \textrm{(VSO4)} &\Longleftrightarrow \Phi_{ij} \textrm{ satisfies (SO4) if } i = j \textrm{ and (SO4') if } i \ne j \\
\Phi\textrm{ satisfies } \textrm{(VSO4')} &\Longleftrightarrow \textrm{each }\Phi_{ij} \textrm{ satisfies (SO4')}
\end{align*}
and similarly for the uniform variants of these conditions. In conjunction with the sheaf property of the space of test objects, this allows one to easily construct global test objects for the vector case from local scalar test objects as follows: suppose we are given an atlas $(U_\alpha, \varphi_\alpha)$ of $M$ and choose $\Phi^\alpha_{ij} \in \SO {\varphi_\alpha(U_\alpha)}$ satisfying (SO1), (SO2) if $i=j$ and (SO2') if $i \ne j$, (SO3), (SO4) if $i=j$ and (SO4') if $i \ne j$. This defines test objects $\Phi^\alpha \in \TOv{U_\alpha, E}$ which, using Propositions \ref{sk_sheaf} and \ref{skproploc} and a partition of unity, can be glued together to give a test object $\Phi \in \TOv{E}$.

In particular, rewriting the embedding of the special Colombeau algebra in a suitable way one obtains a net of smoothing operators (given by convolution) satisfying (SO1)--(SO4), cf.~\cite[Section 9]{papernew}. In this way one can also obtain properties like (SO2) for other spaces of distributions, e.g.,\ $C^k$-functions for $k < \infty$, because an analogoue of Proposition \ref{skproploc} holds for these properties.

More geometrically, vector smoothing operators may be defined by parallel transport along geodesics (locally in geodesically convex neighborhoods) if there is a background connection avaliable on the manifold (see \cite{rm_diss} for more details). In fact, this is a special case of the combination of so-called transport operators (sections of the external tensor product $E^* \boxtimes E$) and scalar smoothing kernels using the isomorphism \eqref{vecblahiso}; such combinations have been employed more generally in \cite{global2}.

\subsection{Classes of smoothing operators}\label{to_classes}

In certain situations it might be desirable to incorporate specific information about a given problem into the definition of the generalized function spaces used to formulate and solve the problem. Situations where this occurs are for example if the underlying manifold has additional structure like an orientation, symmetries, a background connection etc., or if physical considerations dictate certain properties of the generalized functions considered.

For these reasons it is sensible to consider classes of test objects which satisfy additional properties to those listed above. The present literature on Colombeau algebras supplies plenty of possible variations on test objects (both for full and special algebras) of which we do not even attempt to initiate a systematic study here. Instead, it shall suffice to point out how the choice of test objects is reflected in (i) the sheaf property, and (ii) invariance under diffeomorphisms.

For (i), the sheaf property rests on Proposition \ref{skproploc}, which needs to be valid for any additional properties we might impose on test objects. In other terms, these properties should be defined in terms of a suitable \emph{local} behaviour if we want to obtain a sheaf.

For (ii), invariance of the spaces of generalized functions under diffeomorphisms obviously rests on the corresponding property of the spaces of test objects. This question can of course be posed also for more particular classes of transformations instead of all diffeomorphisms.

One example of how the test objects can be modified according to the situtation is given by the space of generalized tensor fields $\hat\cG^r_s$ of \cite{global2}. There, only tensor bundles $T^r_sM$ of the base manifold $M$ are considered. The isomorphism \eqref{vecblahiso} in conjunction with the fact that a transport operator on $TM$ induces a transport operator on each $T^r_sM$ (cf.\ \cite[Definition 6.3]{global2}) then gives, in principle, rise to a theory of generalized tensor fields having only one space of scalars, in contrast to the general theory which is the scope of this article. However, $\hat\cG^r_s$ lacks the sheaf property and the possibility to define $\Liet_X$ and generalized covariant derivatives. However, the use of transport operators could be combined with the structure outlined in this article to obtain an improved version of $\hat\cG^r_s$ having these properties.

The case of manifolds with given background connections to be used in the embedding of tensorial distributions has been studied in \cite{rm_diss}; in this situation, the embedding has been shown to commute with homotheties and Lie derivatives along Killing vector fields.

We finally remark that due to our functional analytic approach and the detailed study of the sheaf property above, it is comparatively easy to adapt and modify the theory outlined in this article as desired.

\section{The quotient construction}

As is typical for spaces of nonlinear generalized functions in the sense of Colombeau, the basic spaces of Definition \ref{def_basicspace} need to undergo a quotient construction in order to preserve the tensor product of smooth sections. For this purpose one singles out subsets of so-called moderate and negligible elements of $\Gb E\Delta$. The respective definitions are taylored in a way such that they are invariant under the Lie derivatives of Definition \ref{deflieder}. Furthermore, for more generality the tests incorporate uniform sets of test objects (Definition \ref{def_uniform}); this is necessary for the proof of Theorem \ref{nonegder} but will be seen to be equivalent to using single test objects only.

\begin{definition}\label{def_testing}$R \in \Gb E\Delta$ is called \emph{moderate} if $\forall \csna \in \csn{\Gamma(E)}$ $\forall j \in \bN_0$ $\exists N \in \bN$ such that for all uniform subsets $A \subseteq \TOv\Delta$ and $B \subseteq \TOvz\Delta$ we have
 \[ \csna( ( \ud^j R) (\Phi_\e) (\Psi_{1,\e}, \dotsc, \Psi_{j,\e}) ) = O(\e^{-N}) \]
uniformly for $(\Phi_\e)_\e \in A$ and $(\Psi_{i,\e})_\e \in B$ for $i=1 \dotsc j$. The subset of moderate elements of $\Gb E\Delta$ is denoted by $\Gm E\Delta$.

$R \in \Gb E\Delta$ is called \emph{negligible} if $\forall \csna \in \csn{\Gamma(E)}$ $\forall j \in \bN_0$ $\forall m \in \bN_0$ for all uniform subsets $A \subseteq \TOv\Delta$ and $B \subseteq \TOvz\Delta$ we have
 \[ \csna( ( \ud^j R) ( \Phi_{\e}) ( \Psi_{1,\e}, \dotsc, \Psi_{j,\e}) ) = O(\e^m) \]
uniformly for $(\Phi_{\e})_\e \in A$ and $(\Psi_{i,\e})_\e \in B$ for $i=1 \dotsc j$. The subset of negligible elements of $\Gb E\Delta$ is denoted by $\Gn E\Delta$. 
\end{definition}

Note that because $\Gamma(E)$ carries the projective topology with respect to all restrictions $\Gamma(E) \to \Gamma(U, E)$ for $U$ from an open cover of $M$, in practice it suffices to estimate $\csna ( (\ud^j R)(\Phi_\e)(\Psi_{1,\e}, \dotsc, \Psi_{j,\e})|_U )$ for $\csna \in \csn{\Gamma(U, E)}$. In particular, we will often assume $U$ to be relatively compact. Moreover, because the $j$-th differential is symmetric (\cite[5.11]{KM}) one can test with $\Psi_{1,\e} = \dotsc = \Psi_{j,\e}$.

We have the following simplification of these definitions.

\begin{theorem}\label{simplemod}
 $R \in \Gb E\Delta$ is moderate if and only if $\forall \csna \in \csn{\Gamma(E)}$ $\forall j \in \bN_0$ $\exists N \in \bN$ $\forall (\Phi_{\e})_\e \in \TOv\Delta$, $(\Psi_{1,\e})_\e$, $\dotsc$, $(\Psi_{j,\e})_\e \in \TOvz\Delta$ we have
\[ \csna (( \ud^j R)(\Phi_\e)(\Psi_{1,\e},\dotsc,\Psi_{j,\e})) = O(\e^{-N}). \]
Similarly, $R$ is negligible if and only if $\forall \csna \in \csn{\Gamma(E)}$ $\forall j \in \bN_0$ $\forall m \in \bN$ $\forall (\Phi_{\e})_\e \in \TOv\Delta$, $(\Psi_{1,\e})_\e$, $\dotsc$, $(\Psi_{j,\e})_\e \in \TOvz\Delta$ we have
\[ \csna (( \ud^j R)(\Phi_\e)(\Psi_{1,\e},\dotsc,\Psi_{j,\e})) = O(\e^m). \]
\end{theorem}
\begin{proof}
Assuming that the condition for moderateness holds, suppose $R$ is not moderate in the sense of Definition \ref{def_testing}. This means that there exist $\csna \in \csn{\Gamma(E)}$ and $j \in \bN_0$ such that for $N \in \bN$ obtained from the assumption for this $\csna$ and $j$, there are uniform subsets $A \subseteq \TOv\Delta$ and $B \subseteq \TOvz\Delta$ such that
\begin{multline*}
 \forall C>0\ \forall \e_0 \in I\ \exists \eta \le \e_0\ \exists (\Phi_\e)_\e \in A\ \exists (\Psi_{i,\e})_\e \in B\ (i=1\dotsc j): \\
\csna((\ud ^j R)(\Phi_\eta) ( \Psi_{1,\eta},\dotsc,\Psi_{j,\eta})) > C \cdot \eta^{-N}.
\end{multline*}
From this we obtain a strictly decreasing sequence $(\e_n)_n \searrow 0$ and sequences $((\Phi^n_\e)_\e)_n$ in $A$, $((\Psi^n_{i,\e})_\e)_n$ in $B$ $(i=1 \dotsc j)$ such that
\begin{equation}\label{bulbber}
\csna((\ud^jR)(\Phi_{\e_n}^n)(\Psi^n_{1,\e_n},\dotsc,\Psi_{j,\e_n}^n))>n \cdot \e_n^{-N} \qquad (n \in \bN).
\end{equation}
Choose any $(\widetilde \Phi_\e)_\e \in \TOv\Delta$ and $(\widetilde \Psi_{i,\e})_\e \in \TOvz\Delta$ such that $\widetilde \Phi_{\e_n} = \Phi^n_{\e_n}$ and $\widetilde \Psi_{i,\e_n} = \Psi^n_{i,\e_n}$ for all $n$, which is possible because $A$ and $B$ are uniform. \proofs{\#25}
For example, set $\widetilde\Phi_\e \coleq \Phi_\e^n$ and $\widetilde\Psi_{i,\e} \coleq \Psi^n_{i,\e}$ for $\e_{n+1} < \e \le \e_n$ with $\e_0 \coleq 1$.
By assumption we then have
 \[ \exists C>0\ \exists \e_0\ \forall \e\le \e_0: \csna((\ud^j R)(\widetilde \Phi_\e)(\widetilde \Psi_{1,\e},\dotsc,\widetilde \Psi_{j,\e})) \le C \cdot \e^{-N}. \]
For $n \in \bN$ such that $n \ge C$ and $\e_n \le \e_0$ this gives a contradiction to \eqref{bulbber}, hence $R$ must be moderate.
The proof for negligibility goes analogously.
\end{proof}

A classical result in Colombeau algebras (\cite[Theorem 1.2.3]{GKOS}) states that negligibility of moderate elements can be tested without resorting to derivatives. This is true also in our setting:

\begin{theorem}\label{nonegder}$R \in \Gm E\Delta$ is negligible if $\forall K \csub M$ $\forall m \in \bN$ $\forall (\Phi_{\e})_{\e} \in \TOv\Delta$:
\begin{equation}\label{alphaa}
\sup_{x \in K} \norm { R ( \Phi_{\e} ) } = O(\e^m)
\end{equation}
where $\norm{\cdot}$ is the norm on $\Gamma(E)$ induced by any Riemannian metric on $E$.
\end{theorem}

\begin{proof}
Suppose $R \in \Gm E\Delta$ satisfies the negligibility test of Theorem \ref{simplemod} for $j=j_0 \in \bN_0$. Testing for negligibility of $R$ with $j = j_0 + 1$, fix $\csna \in \csn{\Gamma(E)}$ and $m \in \bN$. Then by moderateness of $R$ there exists $N \in \bN$ such that for all $(\Phi_\e)_\e \in A$ and $(\Psi_{1,\e})_\e, \dotsc, (\Psi_{j_0+2, \e})_\e \in B$ with $A \subseteq \TOv\Delta$ and $B \subseteq \TOvz\Delta$ uniform we have
\[ \csna((\ud^{j_0+2} R)(\Phi_\e)(\Psi_{1,\e}, \dotsc, \Psi_{j_0+2,\e})) = O(\e^{-N}). \]
At the same time, by assumption we have
\[ \csna ( (\ud^{j_0} R)(\Phi_\e)(\Psi_{1,\e}, \dotsc, \Psi_{j_0, \e})) = O(\e^{2m + N}). \]
In order to estimate $\csna((\ud^{j_0+1} R)(\Phi_\e)(\Psi_{1,\e}, \dotsc, \Psi_{j_0+1, \e}))$ for given $(\Phi_\e)_\e$, $(\Psi_{1,\e})_\e$, $\dotsc$, $(\Psi_{j_0+1, e})_\e$ we use the fact that
\[
(\ud^{j_0} R)(\Phi_\e + \e^{m + N}\Psi_{j_0+1,\e})(\Psi_{1,\e}, \dotsc, \Psi_{j_0, \e})
\] equals
\begin{multline*}
(\ud^{j_0} R)(\Phi_\e)(\Psi_{1,\e}, \dotsc, \Psi_{j_0, \e}) + \e^{m + N} (\ud^{j_0+1} R)(\Phi_\e)(\Psi_{1,\e}, \dotsc, \Psi_{j_0, \e}, \Psi_{j_0+1, \e}) \\
+ \e^{2m + 2N} \int_0^1 (1-t)(\ud^{j_0 + 2} R)(\Phi_\e + t \e^{m + N} \Psi_{j_0 + 1,\e}) \\
\cdot (\Psi_{1,\e}, \dotsc, \Psi_{j_0,\e}, \Psi_{j_0+1,\e}, \Psi_{j_0+1,\e})\,\ud t.
\end{multline*}
As $\{ (\Phi_\e + t \e^{m + N} \Psi_{j_0+1, \e})_\e \ |\ t \in (0,1) \}$ is a uniform set of test objects, $\csna((\ud^{j_0+1} R)(\Phi_\e)(\Psi_{1,\e}, \dotsc, \Psi_{j_0+1, \e})) = O(\e^{m})$ follows from the above. Inductively, we see that $R \in \Gm E \Delta$ is negligible already if the test of Theorem \ref{simplemod} holds for $j=0$.

Because $\Gamma(E)$ carries the projective topology with respect to all mappings $\pr_i \circ |_U: \Gamma(E) \to \Gamma(U, E) \cong C^\infty(U)^k \to C^\infty(U)$ with $\dim E = k$ for charts $(U, \varphi)$ of $M$ such that $E|_U$ is trivial, we may in fact assume that $R \in C^\infty(\VSO \Delta, C^\infty(U))$ for some chart $(U, \varphi)$. Because $C^\infty(U) \cong C^\infty(\varphi(U))$, our claim is established if we can show the local estimate
\begin{multline}\label{equazione}
\forall K \csub \varphi(U)\ \forall \alpha \in \bN_0^n\ \forall m \in \bN\ \forall (\Phi_\e)_\e \in \TOv\Delta:\\
\sup_{x \in K} \abso{ \pd_x^\alpha ( R ( \Phi_\e ) \circ \varphi^{-1})(x)} = O (\e^m).
\end{multline}
Fix $K$ and $m$. As above, the assumption that $R$ is moderate and satisfies \eqref{alphaa} translates into
\begin{align*}
 \exists N\ &\forall (\Phi_\e)_\e: \sup_{x \in K} \abso{\pd_x^{\alpha_0 + 2 e_i} ( R ( \Phi_\e) \circ \varphi^{-1})(x)} = O(\e^{-N}) \\
& \forall (\Phi_\e)_\e: \sup_{x \in K} \abso{ \pd_x^{\alpha_0} ( R ( \Phi_\e) \circ \varphi^{-1})(x)} = O(\e^{2 m + N} )
\end{align*}
for $\alpha_0 = 0$. For small $\e$ and $x \in K$, we have the expansion
\begin{multline*}
 \pd_x^{\alpha_0} ( R( \Phi_\e) \circ \varphi^{-1})(x + \e^{m + N}e_i) \\
= \pd_x^{\alpha_0} ( R(\Phi_\e) \circ \varphi^{-1}) (x) + \e^{m + N} \pd_x^{\alpha_0 + e_i} ( R(\Phi_\e)\circ \varphi^{-1})(x) \\
+ \e^{2m + 2N} \int_0^1 (1-t) \pd_x^{\alpha_0 + 2 e_i}(R(\Phi_\e) \circ \varphi^{-1} )(x + t \e^{m + N} e_i) \, \ud t
\end{multline*}
where $e_i$ is the $i$-th Euclidean basis vector of $\bR^n$. From this $\pd_x^{\alpha_0 + e_i} ( R(\Phi_\e) \circ \varphi^{-1}) = O(\e^{m})$ on $K$ follows, which gives \eqref{equazione} by induction and hence shows the claim.
\end{proof}

We will now combine the basic space and test objects in order to define the Colombeau quotient. Although the following theorem is a cornerstone of the construction its proof is a trivial consequence of the definitions, which shows that our functional analytic approach indeed is very natural.
\begin{theorem}
 \begin{enumerate}[label=(\roman*)]
  \item $\iota(\cD'(M,E)) \subseteq \Gm E{\{E\}}$.
  \item $(\iota - \id)(\Gamma(E)) \subseteq \Gn E{\{E\}}$. Hence, for $s \in \Gamma(E)$ and $t \in \Gamma(F)$, $\iota(s) \otimes \iota(t) - \iota(s \otimes t) \in \Gn E \Delta$, with $\Delta = \{ E, F, E \otimes F \}$.
  \item $\iota(\cD'(M,E)) \cap \Gn E{\{E\}} = \{ 0 \}$.
 \end{enumerate}
\end{theorem}

The following is easily seen from the definitions.
\begin{proposition}\label{modalgprop}
\begin{enumerate}[label=(\roman*)]
 \item $\Gn E\Delta$ is a $C^\infty(M)$-submodule of $\Gm E\Delta$, which itself is a $C^\infty(M)$-submodule of $\Gb E\Delta$.
\item The tensor product \eqref{genmultdef} maps $\Gm E{\Delta} \times \Gm F{\Delta}$ into $\Gm{E \otimes F}{\Delta}$. If one of the factors is negligible their tensor product is so.
 \item \label{modalgprop.3} The isomorphisms of Theorem \ref{reprmod} preserve moderateness and negligibility, i.e., a generalized section is moderate or negligible if and only if all its components are:
\begin{gather*}
\Gm E\Delta \cong \Gm M\Delta \otimes_{C^\infty(M)} \Gamma(E) \cong \Hom_{C^\infty(M)}(\Gamma(E^*), \Gm M\Delta) \\
\Gn E\Delta \cong \Gn M\Delta \otimes_{C^\infty(M)} \Gamma(E) \cong \Hom_{C^\infty(M)}(\Gamma(E^*), \Gn M\Delta)
\end{gather*}
\item $\Lieh_X$ and $\Liet_X$ preserve moderateness and negligiblity:
\begin{gather*}
 \Lieh_X \Gm E\Delta \subseteq \Gm E\Delta, \qquad \Lieh_X \Gn E\Delta \subseteq \Gn E\Delta, \\
 \Liet_X \Gm E\Delta \subseteq \Gm E\Delta, \qquad \Liet_X \Gn E\Delta \subseteq \Gn E\Delta.
\end{gather*}
\item Vector bundle isomorphisms preserve moderateness and negligibility:
\[ \mu_*(\Gm E\Delta) \subseteq \Gm {E'}{\Delta'},\qquad \mu_*(\Gn E\Delta) \subseteq \Gn {E'}{\Delta'}, \]
where $\mu = \{\mu_G\colon G \to G'\}_{G \in \Delta \cup \{E\}}$ is a family of vector bundle isomorphisms over the same diffeomorphism and $\Delta' = \{ G' \}_{G \in \Delta}$.
\end{enumerate}
\end{proposition}
\begin{proof}
(i) is clear. For (ii), note that $\ud^j (R \otimes S)(\Phi)(\Psi, \dotsc, \Psi)$ equals the sum $\sum_{l=0}^j \binom{j}{l} \ud ^l R(\Phi)(\Psi, \dotsc, \Psi) \otimes \ud^{j-l} S(\Phi)(\Psi, \dotsc, \Psi)$. Because $\otimes: \Gamma(E) \times \Gamma(F) \to \Gamma(E \otimes F)$ is continuous, for each $\csna \in \csn {E \otimes F}$ there exist $\csnb \in \csn E$ and $\csnc \in \csn F$ such that $\csna ( s \otimes t ) \le \csnb(s) \csnc(t)$ $\forall s \in \Gamma(E), t \in \Gamma(F)$, which implies the claim.

For (iii), moderateness and negligibility of the components follows from continuity of the contraction $m_g\colon \Gamma(E) \to C^\infty(M)$ with $g \in \Gamma(E^*)$. For the converse direction, as noted above it suffices to test on sets $U$ where $E$ is trivializable; the claim then follows from $\Gamma(U,E) \cong C^\infty(U)^k$ with $\dim E = k$ because we can estimate seminorms $R(\Phi_\e)|_U$ by seminorms of its finitely many components.

(iv) and (v) are easily seen from the definitions.
\end{proof}

\begin{definition}
The $C^\infty(M)$-quotient module $\Gq E\Delta \coleq \Gm E\Delta / \Gn E\Delta$ is called the \emph{space of generalized sections of $E$} with index set $\Delta$.
\end{definition}

From Theorem \ref{modalgprop} \ref{modalgprop.3} one easily sees that we have induced isomorphisms \proofs{\#45}
\[ \Gq E\Delta \cong \Gq M\Delta \otimes_{C^\infty(M)} \Gamma(E) \cong \Hom_{C^\infty(M)}(\Gamma(E^*), \Gq M\Delta). \]
Moreover, we obtain mappings
\begin{align*}
 \otimes \colon \Gq E\Delta \times \Gq F\Delta &\to \Gq {E \otimes F}\Delta \\
\Lieh_X \colon \Gq E\Delta & \to \Gq E\Delta \qquad (X \in \fX(M)) \\
\Liet_X \colon \Gq E\Delta & \to \Gq E\Delta \qquad (X \in \Gq {TM}\Delta) \\
\mu_* \colon \Gq E\Delta &\to \Gq {E'}{\Delta'}
\end{align*}
with analoguous properties as on the basic space, i.e., $\otimes$ is $\Gq M\Delta$-bilinear, 
the Lie derivatives are $\bK$-bilinear in both arguments, and $\Liet_X$ is $\Gq M\Delta$-linear in $X$. Furthermore, we can define the mixed tensor algebra
\[
 \Gt E\Delta \coleq \bigoplus_{r,s \ge 0} \Gq{E^r_s}\Delta.
\]

\section{Association}\label{sec_association}

The concept of association, which is an equivalence relation in spaces of Colombeau generalized functions which is coarser than equality, is important for modelling a wide range of physical phenomena (cf.~\cite[95]{MOBook}). It extends faithfully the concept of equality of distributions, i.e., two embedded distributions are equal if they are associated. Furthermore, products of smooth functions and distributions, which are not preserved in Colombeau algebras on the level of equality, are preserved on the level of association.
It has furthermore been observed, roughly said, that whenever calculations make sense in distribution theory, the analogue calculations in Colombeau algebras give a result associated to the classical result. Thus, this concept is the means by which full compatibility of Colombeau algebras with distribution theory is obtained (cf.~\cite{ColElem}).

In this section we will give the definition and main properties of association of generalized sections.

Note that by the Banach-Steinhaus theorem, convergence of a net $(u_\e)_{\e \in (0,1]}$ in $\cD'(M, E)$ is equivalent to weak convergence, i.e.,
\[ u_\e \to 0 \textrm{ in }\cD'(M,E) \Longleftrightarrow \forall s \in \Gamma(E^* \otimes \Vol(M)): \langle u_\e, s \rangle \to 0. \]
Furthermore, this is equivalent to componentwise convergence, i.e., $u_\e \to 0$ in $\cD'(M,E)$ if and only if $u_\e \cdot \SDa \to 0$ in $\cD'(M)$ for all $\SDa \in \Gamma(E^*)$, or equivalently, $\langle u_\e \cdot \SDa, \omega \rangle \to 0$ for all $\SDa \in \Gamma(E^*)$ and $\omega \in \Gamma(\Vol(M))$.

\begin{definition}\label{def_assoc}
\begin{enumerate}[label=(\roman*)]
\item $R,S \in \Gb E\Delta$ are called \emph{associated} (written $R \approx S$) if $\forall (\Phi_\e)_\e \in \TOv\Delta$: $R(\Phi_\e) - S(\Phi_\e) \to 0$ in $\cD'(M, E)$.
\item $R \in \Gb E\Delta$ is said to admit $u \in \cD'(M,E)$ as an \emph{associated distribution} if $R \approx \iota(u)$, which is the case if and only if $\forall (\Phi_\e)_\e \in \TOv\Delta$: $R(\Phi_\e) \to u$ in $\cD'(M, E)$.
\end{enumerate}
\end{definition}

Clearly any $R \in \Gn E\Delta$ is associated to $0$, hence association is well-defined on $\Gq E\Delta$ by defining two elements of $\Gq E\Delta$ to be associated if any of their respective representatives are.

\begin{proposition}\label{prop_assoc}
 \begin{enumerate}[label=(\roman*)]
  \item \label{prop_assoc.1} For $r \in \Gamma(E)$ and $u \in \cD'(M, F)$, $r \otimes \iota(u) \approx \iota (r \otimes u)$.
  \item\label{prop_assoc.2} For $u \in \cD'(M, E)$ and $\SDa \in \Gamma(E^*)$, $\iota(u) \cdot \SDa \approx \iota(u \cdot \SDa)$.
  \item Let $R,S \in \Gq E\Delta$. Then $R \approx S$ if and only if $R \cdot \SDa \approx S \cdot \SDa$ $\forall \SDa \in \Gamma(E^*)$.
  \item If $R_i, S_i \in \Gq E\Delta$ such that $R_i \approx S_i$ for $i = 1,2$ then $R_1 + R_2 \approx S_1 + S_2$.
  \item If $R,S \in \Gq E\Delta$ such that $R \approx S$ then $s \otimes R \approx s \otimes S$ $\forall s \in \Gamma(F)$.
  \item\label{prop_assoc.6} If $R,S \in \Gq E\Delta$ such that $R \approx S$, then $\Liet_X R \approx \Liet_X S$ $\forall X \in \fX(M)$.
\item\label{prop_assoc.7} $\Lieh_X (\iota u) \approx \Liet_X ( \iota u)$ $\forall u \in \cD'(M,E)$ $\forall X \in \fX(M)$.
 \end{enumerate}
\end{proposition}
\begin{proof}
 (i): Fix $s \otimes t \in \Gamma(E^*) \otimes_{C^\infty(M)} \Gamma(F^* \otimes \Vol(M))$ and $(\Phi_\e)_\e \in \TOv F$. Then
\begin{align*}
 \langle (r \otimes \iota(u))(\Phi_\e), s \otimes t \rangle &= \langle r \otimes \iota(u)(\Phi_\e), s \otimes t \rangle = \langle \iota(u)(\Phi_\e), (r \cdot s) t \rangle \\
& \to \langle u, (r \cdot s) t \rangle = \langle r \otimes u, s \otimes t \rangle.
\end{align*}

\ref{prop_assoc.2}--\ref{prop_assoc.6} are easily verified in an elementary manner.

\ref{prop_assoc.7}: $(\Lieh_X(\iota u) - \Liet_X(\iota u))(\Phi_\e) = - (\iota u)(\Lso_X \Phi_\e) + \Lie_X((\iota u)( \Phi_\e)) - \Lie_X ( (\iota u) ( \Phi_\e)) = \Phi_\e ( \Lie_X (u)) - \Lie_X ( \Phi_\e(u)) \to \Lie_X u - \Lie_X u = 0$ in $\cD'(M,E)$ by (VSO2) and continuity of $\Lie_X$.
\end{proof}

\begin{remark} In Colombeau theory, there are many distinguished forms of association. These come about by two generalizations of Definition \ref{def_assoc} (i). The first of these is to replace convergence in $\cD'(M,E)$ by convergence in some other space $\cH$ of distributions, for example the space of $C^k$-tensor fields for $k \in \bN$. For results analogous to Proposition \ref{prop_assoc} \ref{prop_assoc.1} to hold, one then also has to adapt the space of test objects $\TOv E$ such that its elements converge to the identity in $\Lincb(\cH, \cH)$, which poses no difficulty in many cases. The second generalization in the definition of association is to distinguish various speeds of covergence of $R(\Phi_\e)$. 
\end{remark}

\section{The sheaf property}

In a geometrical context the sheaf property is essential because this is what allows one to talk of local coordinates uniquely defining a generalized section. Throughout this section, let a manifold $M$, a vector bundle $E$ and an index set $\Delta$ be fixed.

Although the basic spaces $\Gb E\Delta$ are sufficiently large for defining all desired operations like tensor products and covariant derivatives for embedded vector-valued distributions, it is sometimes desirable to work in smaller subspaces which have better properties or are easier to handle.

The motivation for the introduction of so-called locality properties of nonlinear generalized functions originally was to obtain the sheaf property. As was seen in the scalar case on $\Omega \subseteq \bR^n$ (cf.~\cite{papernew}), the entire basic space $\Gnb\Omega \coleq C^\infty(\SK \Omega, C^\infty(\Omega))$ is too large in order to give a sheaf after the Colombeau quotient construction because for $R \in \Gnb\Omega$, $\vec\varphi \in \SK \Omega$ and $x \in \Omega$, the expression $R(\vec\varphi)(x)$ may depend on the behaviour of $\vec\varphi$ at points far away from $x$.

The basic observation then was that for $u \in \cD'(\Omega)$, $(\iota u)(\vec\varphi)(x)$ depends only on $\vec\varphi(x)$, which is the strongest locality condition. Relaxing it to various degress, one obtains subalgebras $\cE_{pi}(\Omega) \subseteq \cE_{ploc}(\Omega) \subseteq \cE_{loc}(\Omega) \subseteq \cE(\Omega)$, defined by simple algebraic conditions, with the remarkable properties that (\cite{papernew})
\begin{enumerate}[label=(\roman*)]
 \item $\cE_{pi}(\Omega) \cong C^\infty(\cD(\Omega))$, recovering Colombeau's original basic space of \cite{ColNew},
 \item $\cE_{ploc}(\Omega) \cong C^\infty(\cD(\Omega), C^\infty(\Omega))$, recovering the basic space of $\Gd$,
 \item $\cE_{loc}(\Omega)$ provides exactly what is needed for the Colombeau quotient to be a sheaf.
\end{enumerate}

In order to obtain the sheaf property in the vector valued setting we will transfer the notion of locality to the formalism of smoothing operators. Because of the structure of the basic space $\Gb E\Delta$ there is no simple translation of the concepts of point-independence and point-locality from \cite{papernew}.

For any vector bundle $E \to M$, $\Phi \in \VSO E$, $g \in \Gamma(E^*)$ and $p \in M$, $m_g \colon \Gamma(E) \to C^\infty(M)$ denotes contraction with $g$ and $\ev_p\colon \Gamma(E) \to \bR$ evaluation at $p$, hence $(\ev_p \circ m_g \circ \Phi)(u) = (\Phi(u) \cdot g)(p)$ for $u \in \cD'(M, E)$. For $\Phi \in \VSO \Delta$ we define $\ev_p$ and $m_g$ componentwise. 

\begin{definition}\label{def_locprop}$R \in \Gb E\Delta$ is called \emph{local} if for all $\Phi,\Phi' \in \VSO \Delta$ and open subsets $U \subseteq M$, $|_U \circ \Phi = |_U \circ \Phi'$ implies $R(\Phi)|_U = R(\Phi')|_U$.
We denote by $\Gbloc E\Delta$ the $C^\infty(M)$-submodule of $\Gb E\Delta$ consisting of local elements.
\end{definition}
Obviously, $\iota$ maps $\cD'(M, E)$ into $\Gbloc E{\{ E \}}$ and $\Gamma(E) \subseteq \Gbloc E\emptyset$. Locality is preserved by diffeomorphisms, tensor products and contraction:

\begin{proposition}
\begin{enumerate}[label=(\roman*)]
 \item With the notation of Definition \ref{def_vbiso}, $\mu_*(\Gbloc E\Delta) \subseteq \Gbloc{E'}{\Delta'}$.
 \item For $R \in \Gbloc E{\Delta}$ and $S \in \Gbloc F{\Delta}$ we have $R \otimes S \in \Gbloc{E \otimes F}{\Delta}$.
 \item $\Gbloc E\Delta \cong \Gbloc M\Delta \otimes_{C^\infty(M)} \Gamma(E) \cong \Hom_{C^\infty(M)}(\Gamma(E^*), \Gbloc M\Delta)$ as $C^\infty(M)$-modules.
\end{enumerate}
\end{proposition}

The proof is trivial and thus omitted. Next, we will consider Lie derivatives.

\begin{proposition} Let $R \in \Gbloc E\Delta$. Then for $X \in \fX(M)$, $\Lieh_X R \in \Gbloc E\Delta$ and for $X \in \Gbloc {TM}\Delta$, $\Liet_X \in \Gbloc E\Delta$.
\end{proposition}
\begin{proof}
Suppose $X \in \Gbloc {TM}\Delta$, $U \subseteq M$ is open and $\Phi,\Phi' \in \VSO \Delta$ are such that $|_U \circ \Phi = |_U \circ \Phi'$. Then $R(\Phi)|_U = R(\Phi')|_U$, $X(\Phi)|_U = X(\Phi')|_U$, and consequently $(\Liet_X R)(\Phi)|_U = (\Lie_{X(\Phi)} R(\Phi))|_U = \Lie_{X(\Phi)|_U} R(\Phi)|_U = \Lie_{X(\Phi')|_U} R(\Phi')|_U = (\Lie_{X(\Phi')} R(\Phi'))|_U = (\Liet_X R)(\Phi')|_U$, hence $\Liet_X R$ is local.

Now suppose that $X \in \fX(M)$. Then $|_U \circ \Lso_X\Phi = |_U \circ \Lie_X \circ \Phi - |_U \circ \Phi \circ \Lie_X = |_U \circ \Lie_X \circ \Phi' - |_U \circ \Phi' \circ \Lie_X = |_U \circ \Lso_X\Phi'$ and thus
$\ud R(\Phi)( \Lso_X \Phi)|_U = (\frac{\ud}{\ud t}|_{t=0} R(\Phi + t \cdot \Lso_X \Phi))|_U = \frac{\ud}{\ud t}|_{t=0} ( R(\Phi + t \cdot \Lso_X \Phi)|_U ) = \frac{\ud}{\ud t}|_{t=0} ( R(\Phi' + t \cdot \Lso_X \Phi'))|_U) = \ud R(\Phi')(\Lso_X\Phi')|_U$. In sum we have $(\Lieh_XR)(\Phi)|_U = (\Lieh_XR)(\Phi')|_U$.
\end{proof}

\begin{definition}We define the $C^\infty(M)$-modules $\Gmloc E\Delta \coleq \Gm E\Delta \cap \Gbloc E\Delta$, $\Gnloc E\Delta \coleq \Gn E\Delta \cap \Gbloc E\Delta$ and $\Gloc E\Delta \coleq \Gmloc E\Delta / \Gnloc E\Delta$.
\end{definition}

One easily sees
that we have $C^\infty(M)$-module isomorphisms
\[ \Gloc E\Delta \cong \Gloc M\Delta \otimes_{C^\infty(M)} \Gamma(E) \cong \Hom_{C^\infty(M)} ( \Gamma(E^*), \Gloc M\Delta) \]
with an analogue of Corollary \ref{algprop} holding. We have induced mappings $\mu_*\colon \Gloc E\Delta \to \Gloc{E'}{\Delta'}$, $\otimes\colon \Gloc E{\Delta} \times \Gloc F{\Delta} \to \Gloc{E \otimes F}{\Delta}$, $\Lieh_X\colon \Gloc E\Delta \to \Gloc E\Delta$ for $X \in \fX(M)$ as well as $\Liet_X\colon \Gloc E\Delta \to \Gloc E\Delta$ for $X \in \Gloc{TM}\Delta$; $\Gloc E\Delta$ is a $\Gloc M\Delta$-module and $\Gloc E\Delta$ is a $C^\infty(M)$-submodule of $\Gq E\Delta$.

The main point of locality is that it enables one to restrict generalized sections as follows.

\begin{theorem}\label{eloc_restriction}Let $U,V \subseteq M$ be open, $V \subseteq U$ and $R \in \Gbloc{U, E}\Delta$. Then there is a unique element $R|_V \in \Gbloc{V, E}\Delta$ such that for any open subset $W \subseteq V$, $\Phi \in \VSO{V, \Delta}$ and $\Phi' \in \VSO {U, \Delta}$ such that $|_W \circ \Phi \circ |_V = |_W \circ \Phi'$ on $\cD'(U,E)$ we have $R|_V (\Phi)|_W = R(\Phi')|_W$.

If in addition we have $\Psi_i \in \VSO{V, \Delta}$ and $\Psi_i' \in \VSO{U, \Delta}$ such that $|_W \circ \Psi_i \circ |_V = |_W \circ \Psi'_i$ for $i=1\dotsc j \in \bN$, then $\ud^j (R|_V)(\Phi)(\Psi_1,\dotsc,\Psi_j)|_W = (\ud^j R)(\Phi')(\Psi_1',\dotsc,\Psi_j')|_W$.

For $f \in C^\infty(U)$, $(f \cdot R)|_V = f|_V \cdot R|_V$. Furthermore, for $U,V,W \subseteq M$ open with $W \subseteq V \subseteq U$,  $(R|_V)|_W = R|_W$, hence $U \mapsto \Gbloc{U, E}\Delta$ is a presheaf of $C^\infty$-modules.
\end{theorem}
\begin{proof}
 Exactly as in the scalar case (cf.~\cite{papernew}).
\end{proof}

\begin{remark}The condition $\Phi(u|_V)|_W = \Psi(u)|_W$ $\forall u \in \cD'(U,E)$ or equivalently $|_W \circ \Phi \circ |_V = |_W \circ \Psi$ would simply correspond to $\vec\varphi|_W = \vec\psi|_W$ in the formalism of smoothing kernels. However, the implicit inclusion $\cD(V) \subseteq \cD(U)$ corresponds to defining $\Phi(u) \coleq \Phi(u|_V)$ for $\Phi \in \VSO{V, E}$ and $u \in \cD'(U,E)$, from which we refrain because it would be too ambigous notationally.
\end{remark}

\begin{lemma}\label{tensor_restrict}Let $U,V$ be open with $V \subseteq U$.
\begin{enumerate}[label=(\roman*)]
 \item \label{tensor_restrict.1} For $R \in \Gbloc{U, E}\Delta$ and $S \in \Gbloc {U, F}\Delta$, $(R \otimes S)|_V = R|_V \otimes S|_V$.
 \item \label{tensor_restrict.2} For $R \in \Gbloc{U, E}\Delta$ and $\mu$ as in Definition \ref{def_vbiso}, $\mu_*(R|_V) = (\mu_*R)|_{\mu(V)}$.
 \item \label{tensor_restrict.3} Let $R \in \Gbloc{U, E}\Delta$. For $X \in \fX(U)$, $(\Lieh_X R)|_V = \Lieh_{X|_V} R|_V$. For $X \in \Gbloc{U, TM}\Delta$, $(\Liet_X R)|_V = \Liet_{X|_V} R|_V$.
 \item \label{tensor_restrict.4} For $X,Y \in \Gb{U, TM}\Delta$, $[X, Y]|_V = [X|_V, Y|_V]$.
 \end{enumerate}
\end{lemma}
\begin{proof}\checkedproof 
\ref{tensor_restrict.1}:
let $W$ be open such that $\overline{W} \subseteq V$ is compact, $\Phi \in \VSO{V, \Delta}$, and choose $\Phi' \in \VSO{U, \Delta}$ by Lemma \ref{sk_cut} such that $|_W \circ \Phi \circ |_V = |_W \circ \Phi'$. Because this implies $R|_V (\Phi)|_W = R(\Phi')|_W$ and $S|_V(\Phi)|_W = S(\Phi')|_W$, we obtain $(R|_V \otimes S|_V)(\Phi)|_W = R|_V(\Phi)|_W \otimes S|_V(\Phi)|_W = R(\Phi')|_W \otimes S(\Phi')|_W = (R \otimes S)(\Phi')|_W$, which by Theorem \ref{eloc_restriction} gives the claim.

\ref{tensor_restrict.2} -- \ref{tensor_restrict.4} are proven similarly.
\end{proof}

The statements of Lemma \ref{tensor_restrict} are also valid for the quotient $\Gloc{U, E}\Delta$.

We next show that moderateness and negligibility are local properties.

\begin{proposition}\label{locmodneg}
Let $U \subseteq M$ be open and $(U_\lambda)_\lambda$ an open cover of $U$. Then $R \in \Gbloc{U,E}\Delta$ is moderate or negligible, respectively, if and only if all $R|_{U_\lambda}$ are so.
\end{proposition}
\begin{proof}
 Let $R \in \Gbloc{U, E}\Delta$.  In order to test $R|_{U_\lambda}$ for moderateness it suffices to know
$(\ud^j R|_{U_\lambda})(\Phi_{0,\e})(\Phi_{1,\e}, \dotsc, \Phi_{j,\e})|_K$
for $j \in \bN_0$, $(\Phi_{0,\e})_\e \in \TOv{U_\lambda, E}$, $(\Phi_{i,\e})_\e \in \TOvz{U_\lambda, E}$ ($i=1 \dotsc j$), and $K$ open and relatively compact in $U_\lambda$. This equals
$(\ud^j R)(\Phi'_{0,\e})(\Phi'_{1,\e}, \dotsc, \Phi'_{j,\e})|_K$
where $(\Phi'_{0,\e})_\e \in \TOv{U,E}$ and $(\Phi'_{i,\e})_\e \in \TOvz{U,E}$ are obtained from Lemma \ref{sk_cut} such that $|_K \circ \Phi_{i,\e} \circ |_{U_\lambda} = |_K \circ \Phi'_{i,\e}$ ($i=0 \dotsc j$).
With this, moderateness or negligibility of $R|_{U_\lambda}$ follows immediately from the corresponding property of $R$.
 
Conversely, assume all $R|_{U_\lambda}$ to be moderate or negligible, respectively. It then suffices to know
$(\ud^j R)(\Phi_{0,\e})(\Phi_{1,\e}, \dotsc, \Phi_{j,\e})|_K$
for $j \in \bN_0$, $(\Phi_{0,\e})_\e \in \TOv{U,E}$, $(\Phi_{i,\e})_\e \in \TOvz{U,E}$ ($i=1 \dotsc j$), where $K$ is open, relatively compact and $\overline{K} \subseteq U_\lambda$ for some $\lambda$. Because this expression equals
\[ (\ud^j R|_{U_\lambda}) ( \rSO_{U_\lambda, U} \Phi_{0,\e})(\rSO_{U_\lambda, U} \Phi_{1,\e}, \dotsc, \rSO_{U_\lambda, U} \Phi_{j,\e})|_K \]
for small $\e$, moderateness or negligibility of $R$ follows.
\end{proof}

\begin{corollary} Restriction descends to $\Gloc E\Delta$ by setting $[R]|_V \coleq [R|_V]$. Therefore, $\Gloc{\_, E}\Delta$ is a presheaf of $\SGloc\Delta$-modules on $M$.
\end{corollary}

\begin{theorem}\label{gsheaf}$\Gloc{\_, E}\Delta$ is a sheaf of $\SGloc\Delta$-modules on $M$.
\end{theorem}
\begin{proof}
Let $U \subseteq M$ be open and $(U_\lambda)_\lambda$ an open cover of $U$.  Suppose we are given $[R_\lambda] \in \Gloc{U_\lambda,E}\Delta$ for each $\lambda$ with $[R_\lambda]|_{U_\lambda \cap U_\mu} = [R_\mu]|_{U_\lambda \cap U_\mu}$ $\forall \lambda,\mu$. We need to define $R \in \Gmloc{U,E}\Delta$ such that $[R]|_{U_\mu} = [R_\mu]$ $\forall \mu$; this $R$ then will be unique with this property because of Proposition \ref{locmodneg}.

Choose a smooth partition of unity $(\mu_\lambda)_\lambda$ on $U$ subordinate to $(U_\lambda)_\lambda$ and for each $\lambda$ a function $\rho_\lambda \in C^\infty(U)$ such that $\rho_\lambda \equiv 1$ on an open neighbourhood of $\supp \mu_\lambda$ and $\supp \rho_\lambda \subseteq U_\lambda$. Then define $R \in C^\infty(\VSO \Delta, \Gamma(E))$ by
\begin{equation}\label{gluesum}
R(\Phi) \coleq \sum_\lambda \mu_\lambda \cdot R_\lambda ( \rSO_{U_\lambda, U}\Phi)
\end{equation}
Smoothness of $R$ is implied by smoothness of all $|_K \circ R$ for $K$ in an open cover of $U$, where we can assume that each $K$ is open and relatively compact. In this case there is a finite set $F$ such that $R(\Phi)|_K$ is given by \eqref{gluesum} with the sum only over $\lambda \in F$, for which smoothness is immediate.

We then have to show that $R$ is local. Fix $\Phi,\Phi' \in \VSO \Delta$ such that $|_W \circ \Phi = |_W \circ \Phi'$, where $W \subseteq U$ is open. Then for each $\lambda$ we have $\ev_p \circ \rSO_{U_\lambda, U}\Phi = \ev_p \circ \rSO_{U_\lambda, U}\Phi'$ for $p \in W \cap \carr \mu_\lambda$ because of Theorem \ref{sk_restriction} \ref{9.1.1}, hence $\mu_\lambda(p) \cdot R_\lambda (\rSO_{U_\lambda, U}\Phi)(p) = \mu_\lambda(p) \cdot R_\lambda ( \rSO_{U_\lambda, U}\Phi') (p)$ for $p \in W$ and $R$ is local.

In order to test $R$ for moderateness it suffices to know
\[ (\ud^j R)(\Phi_{0,\e})(\Phi_{1,\e},\dotsc,\Phi_{j,\e})|_K \]
for $j \in \bN_0$, $(\Phi_{0,\e})_\e \in \TOv{U, \Delta}$, $(\Phi_{i,\e})_\e \in \TOvz{U, \Delta}$ ($i=1 \dotsc j$) and $K$ open and relatively compact with $\overline{K} \subseteq U$. This expression is given by
\[ \sum_{\lambda \in F} \mu_\lambda|_K \cdot (\ud^j R_\lambda)(\rSO_{U_\lambda, U} \Phi_{0,\e} )(\rSO_{U_\lambda, U} \Phi_{1,\e} ,\dotsc,\rSO_{U_\lambda, U} \Phi_{j,\e} )|_{K \cap U_\lambda} \]
for finite $F$.
It suffices to estimate for each $\lambda \in F$ the expression
\[
(\ud^j R_\lambda)(\rSO_{U_\lambda, U} \Phi_{0,\e} )(\rSO_{U_\lambda, U} \Phi_{1,\e} ,\dotsc,\rSO_{U_\lambda, U} \Phi_{j,\e} )|_{\supp \mu_\lambda \cap K}
\]
which is moderate by assumption.

Next, we need to show that $R|_{U_\mu} - R_\mu$ is negligible for all $\mu$. Fix $(\Phi_\e)_\e \in \TOv{U_\mu, \Delta}$ for testing. Suppose we want to test $(R|_{U_\mu} - R_\mu)(\Phi_\e)$ on an open relatively compact set $K$ with $\overline{K} \subseteq U_\mu$. Choose $(\Phi'_\e)_\e \in \TOv{U, \Delta}$ by Lemma \ref{sk_cut} such that $|_K \circ \Phi_\e \circ |_{U_\mu} = |_K \circ \Phi_\e'$, which gives $R|_{U_\mu}(\Phi_\e)|_K = R(\Phi_\e')|_K$ for small $\e$. Furthermore, $R_\mu(\Phi_\e)|_K = R_\mu(\rSO_{U_\mu, U}\Phi'_\e)|_K$ because $|_K \circ \Phi_\e = |_K \circ (\rSO_{U_\mu, U}\Phi'_\e)$ for small $\e$:
let $f \in C^\infty(U)$ with $\supp f \subseteq U_\mu$ and $f \equiv 1$ in an open neighborhood of $K$. Then for $u \in \cD'(U_\mu, E)$ and small $\e$, 
$(\rSO_{U_\mu, U}\Phi'_\e)(u)|_K = \Phi'_\e(f \cdot u)|_K = \Phi_\e((f \cdot u)|_{U_\mu})|_K = \Phi_\e(u)|_K$. Hence, we can write $(R|_{U_\mu}-R_\mu)(\Phi_\e)|_K$ as
\[
\sum_{\lambda \in F} \mu_\lambda|_K \cdot \bigl( R_\lambda ( \rSO_{U_\lambda, U} \Phi'_\e )|_K - R_\mu ( \rSO_{U_\mu, U} \Phi'_\e)|_K\bigr).
\]
Because $R_\lambda ( \rSO_{U_\lambda, U} \Phi'_\e ) |_K - R_\mu ( \rSO_{U_\mu, U} \Phi'_\e )|_K = (R_\lambda|_{U_\lambda \cap U_\mu} - R_\mu|_{U_\lambda \cap U_\mu}) ( \rSO_{U_\lambda \cap U_\mu, U} \Phi'_\e)|_K$ for small $\e$, negligibility follows from the assumption.
\end{proof}

\begin{theorem}The embedding $\iota \colon \cD'(M, E) \to \Gloc E{\{E\}}$ is a sheaf morphism, i.e., for any open subset $M' \subseteq M$ and $u \in \cD'(M,E)$ we have $\iota(u)|_{M'} = \iota(u|_{M'})$.
\end{theorem}
\begin{proof}
Given a relatively compact open subset $K \subseteq M'$ and a test object $(\Phi_\e)_\e \in \TOv {M',E}$, choose by Lemma \ref{sk_cut} a test object $(\Psi_\e)_\e \in \TOv {M', E}$ such that $|_K \circ \Phi_\e \circ |_{M'} = |_K \circ \Psi_\e$ for small $\e$. Then the respective test of Theorem \ref{nonegder} estimates, for $x \in K$,
\begin{align*}
 (\iota(u)|_{M'}(\Phi_\e) - \iota(u|_{M'})(\Phi_\e))|_K &= (\iota(u)(\Psi_\e) - \iota(u|_{M'})(\Phi_\e))|_K \\
&= (\Psi_\e(u) - \Phi_\e(u|_{M'})|_K =0. \qedhere
\end{align*}
\end{proof}

\subsection{Local coordinates}

We will now use the sheaf property of $\Gloc{\_,E}\Delta$ to recover the known coordinate formulas of smooth differential geometry.

Suppose $U \subseteq M$ is open such that $E$ is trivializable over $U$ and $\Gamma(U,E)$ has a basis $(b_i)_i$ with dual basis $(\beta^j)_j$ of $\Gamma(U, E^*)$. Then the coordinates of $R \in \Gbloc E\Delta$ on $U$ are given by $R^i \coleq R|_U \cdot \beta^i \in \Gbloc{U}\Delta$.

Suppose now we are given an open cover $(U_\alpha)_\alpha$ of $M$ such that $E$ is trivial over each $U_\alpha$. Then $R$ has coordinates $R_\alpha^i \in \Gbloc{U_\alpha}\Delta$ ($i=1 \dotsc \dim E$) on each $U_\alpha$, which are enough to characterize moderateness, negligiblity and association as follows:

\begin{proposition}
\begin{enumerate}[label=(\roman*)]
\item\label{coordchar.1} $R \in \Gbloc E \Delta$ is moderate or negligible if and only if all $R_\alpha^i$ are.
\item\label{coordchar.2} $R \approx 0$ if and only if $R_\alpha^i \approx 0$ for all $i$ and $\alpha$.
\end{enumerate}
\end{proposition}
\begin{proof}
\ref{coordchar.1} follows immediately from Proposition \ref{locmodneg} and Proposition \ref{modalgprop} \ref{modalgprop.3}, while \ref{coordchar.2} is proved the same way as Proposition \ref{locmodneg}.
\end{proof}

In order to recover the usual coordinate formulas of smooth differential geometry we have to make some preparations.

\begin{lemma}\label{homloc}Let $R \in \Hom_{C^\infty(U)} ( \Gamma(U, E^*), \Gloc U\Delta)$, $s \in \Gamma(U, E^*)$ and $V \subseteq U$ open. Then $s|_V = 0$ implies $R(s)|_V = 0$.
\end{lemma}
\begin{proof} \proofs{\#13}
 Let $W$ be open with $\overline{W} \subseteq V$ and choose $f \in C^\infty(U)$ such that $\supp f \subseteq V$, $f \equiv 1$ on $W$. Then $s = (1-f) \cdot s$ and $R(s)|_W = R((1-f)s)|_W = (1-f)|_W \cdot R(s)|_W = 0$. Covering $V$ by such sets $W$ we find that $R(s)|_V = 0$ because $\Gloc \_\Delta$ is a sheaf.
\end{proof}

The following is proved in the same manner as Theorem \ref{eloc_restriction}.
\begin{lemma}\label{homrestr}Let $R \in \Hom_{C^\infty(U)} ( \Gamma(U, E^*), \Gloc U\Delta)$, $V \subseteq U$ open. There is a unique element $R|_V \in \Hom_{C^\infty(V)} ( \Gamma(V, E^*), \Gloc V\Delta)$ such that for any open subset $W \subseteq V$, $s \in \Gamma(V, E^*)$ and $\tilde s \in \Gamma(U, E^*)$ such that $s|_W = \tilde s|_W$ we have $R|_V (s)|_W = R (\tilde s)|_W$. Furthermore, $R|_{V_1}|_{V_2} = R|_{V_2}$ for $V_2 \subseteq V_1$.
\end{lemma}

In other words, $\Hom_{C^\infty(\_)} ( \Gamma(\_, E^*), \Gloc \_\Delta)$ is a presheaf on $M$.
\begin{theorem}\label{gqlocpsiso}$\cG_{loc}(\_; E) \cong \Hom_{C^\infty(\_)} ( \Gamma(\_, E^*), \Gloc\_\Delta)$ is a presheaf isomorphism.
\end{theorem}
\begin{proof} \proofs{\#14}
 Let $U,V$ be open with $V \subseteq U$ and $[R] \in \Gloc{U, E}\Delta$. Denoting the respective isomorphism of Theorem \ref{reprmod} by $\varphi_U$, we need to show that $\varphi_U([R])|_V \cdot t = \varphi_V ( [R]|_V) \cdot t$ in $\Gloc{V,E}\Delta$ $\forall t \in \Gamma(V, E^*)$. By Theorem \ref{gsheaf} this is the case if $(\varphi_U([R])|_V \cdot t)|_W = (\varphi_V([R]|_V)\cdot t)|_W$ in $\Gloc{W,E}\Delta$ for all open sets $W$ with $\overline{W} \subseteq V$. Choose $\tilde t \in \Gamma(U,E^*)$ such that $t|_W = \tilde t|_W$. Then by the Lemmas \ref{homloc} and \ref{homrestr},
 \[ (\varphi_U([R])|_V \cdot t)|_W = (\varphi_U([R])|_V \cdot \tilde t|_V)|_W = (\varphi_U([R]) \cdot \tilde t)|_W = [(\varphi_U(R) \cdot \tilde t)|_W] \]
 while on the other hand
 \[ (\varphi_V([R]|_V)\cdot t)|_W = (\varphi_V([R]|_V) \cdot \tilde t|_V)|_W = [(\varphi_V(R|_V) \cdot \tilde t|_V)|_W]. \]
In order to test the difference of these expressions for negligibility fix an open relatively compact set $K$ with $\overline{K} \subseteq W$ and $(\Phi_\e)_\e \in \TOv{W, \Delta}$. Using Lemma \ref{sk_cut} choose
\begin{gather*}
 (\Psi^1_\e)_\e \in \TOv{U; \Delta} \text{ such that }|_K \circ \Phi_\e \circ |_W = |_K \circ \Psi^1_\e \text{ on }\cD'(U, \Delta) \text{ and}\\
 (\Psi^2_\e)_\e \in \TOv{V; \Delta} \text{ such that }|_K \circ \Phi_\e \circ |_W = |_K \circ \Psi^2_\e \text{ on }\cD'(V, \Delta),
\end{gather*}
and hence $|_K \circ \Psi^1_\e = |_K \circ \Phi_\e \circ |_W = |_K \circ \Phi_\e \circ |_W \circ |_V = |_K \circ \Psi^2_\e \circ |_V$. Then
\begin{align*}
 (\varphi_U(R)\cdot \tilde t)|_W(\Phi_\e)|_K - (&\varphi_V(R|_V)\cdot \tilde t|_V)|_W (\Phi_\e) \\
 & = (\varphi_U ( R) \cdot \tilde t)(\Psi^1_\e)|_K - (\varphi_V(R|_V)\cdot \tilde t|_V)(\Psi^2_\e)|_K \\
 &= (R(\Psi^1_\e) \cdot \tilde t)|_K - (R|_V (\Psi^2_\e)\cdot \tilde t|_V)|_K \\
 &= (R(\Psi^1_\e)|_K - R(\Psi^1_\e)|_K) \cdot \tilde t |_K = 0
\end{align*}
Conversely, $\varphi_U^{-1} (h)|_V = \varphi_V^{-1}(\varphi_V(\varphi_U^{-1}(h)|_V)) = \varphi_V^{-1}(\varphi_U(\varphi_U^{-1}(h))|_V) = \varphi_V^{-1}(h|_V)$ for all $h \in \Hom_{C^\infty(U)} ( \Gamma(U, E^*), \Gloc{U, E}\Delta$, which shows the claim.
\end{proof}

Note that also on $U \mapsto \Gloc U\Delta \otimes_{C^\infty(U)} \Gamma(E)$ we have a presheaf structure by setting $(F \otimes s)|_V = F|_V \otimes s|_V$. This way, if $R \in \Gloc{U,E}\Delta$ is given by $R = F \otimes s$ then $R|_V = F|_V \otimes s|_V$, which means:
\begin{proposition}
 $\Gloc{U,E}\Delta \otimes_{C^\infty(U)} \Gamma(U,E) \cong \Gloc{U,E}\Delta$ is a presheaf morphism.
 \end{proposition}
\begin{proof}
 Let $R =F \otimes s$. Then $R|_U = F|_U \otimes s|_U$ if $R|_U \cdot t = (F|_U \otimes s|_U)\cdot t$ for all $t \in \Gamma(U, E)$, which is the case if $(R|_U \cdot t)|_W = (F|_U \cdot s|_U \cdot t)|_W$ for all open $W$ with $\overline{W} \subseteq U$. Choose $\tilde t$ with $\tilde t = t$ on $W$, then both sides are equal to $(F \cdot (s \cdot \tilde t))|_W$. The inverse is a presheaf morphism as in the proof of Theorem \ref{gqlocpsiso}.
\end{proof}

With this we obtain (using the Einstein summation convention):
\begin{corollary}
Let $E$ be trivial over the open set $U \subseteq M$, $(b_i)_i$ a basis of $\Gamma(U,E)$ and $(\beta^i)_i$ the dual basis of $\Gamma(U, E^*)$.
\begin{enumerate}[label=(\roman*)]
 \item For $R \in \Gloc E \Delta$ and $\theta \in \Gloc {E^*} \Delta$,
 \[ R(\theta)|_U = R^i \theta_i \]
 where $(R^i)_i$ and $(\theta_j)_j$ are the coordinates of $R$ and $\theta$ with respect to the given bases, respectively.
 \item If $F$ is another vector bundle which is trivial over $U$ with basis $(\tilde b_j)_j$ and dual basis $(\tilde \beta^j)_j$, $R \in \Gloc E \Delta$ and $S \in \Gloc F \Delta$,
 \[ (R \otimes S)^{ij} = R^i S^j \]
 where $(R \otimes S)^{ij}$, $R^i$ and $S^j$ denote the coordinates of $R \otimes S$, $R$ and $S$ with respect to the bases $(b_i)_i$, $(\tilde b_j)_j$ and $(b_i \otimes \tilde b_j)_{i,j}$ of $\Gamma(U, E)$, $\Gamma(U, F)$ and $\Gamma(U, E \otimes F)$, respectively.
 \item For $X \in \Gloc{TM}\Delta$ and $Y \in \Gloc E\Delta$,
 \[ (\Liet_X Y)|_U = X^i \frac{\pd Y}{\pd_i}. \]
 In particular, if $Y \in \Gloc {TM}\Delta$ then
 \[ (\Liet_X Y)|_U = (X^j \pd_j Y^i - Y^j \pd_j X^i)\pd_i. \]
\end{enumerate}
\end{corollary}

\section{Covariant derivatives}

Because covariant derivatives play a paramount role in many applications of differential geometry it is a principal aim of this article to define them also for generalized sections. In order to preserve the usual objects and rules of calculation of classical \mbox{(pseudo-)}Riemannian geometry we will require that generalized covariant derivatives essentially have the same properties as classical ones.

In particular, we want the curvature of a generalized covariant derivative to be a well-defined generalized tensor by the usual formula
\[ \Riem(\GVa,\GVb)\GSb = \nabla_\GVa \nabla_\GVb \GSb - \nabla_\GVb \nabla_\GVa \GSb - \nabla_{[\GVa,\GVb]} \GSb. \]

We will begin by extending a smooth covariant derivative $\nabla \colon \fX(M) \times \Gamma(E) \to \Gamma(E)$ to a generalized covariant derivative $\widehat\nabla \colon \Gb {TM}\Delta \times \Gb E\Delta \to \Gb E\Delta$. If we want this extension to have similar properties this means in particular that
\begin{enumerate}[label=(\roman*)]
\item \label{condi1} $(\GVa,\GSa) \mapsto \widehat \nabla_\GVa \GSa$ is $\bK$-linear in $\GSa$ and $C^\infty(M)$-linear in $\GVa$,
 \item \label{condi2} $\widehat \nabla$ has an extension to $\Gtb E\Delta$ which satisfies the Leibniz rule and commutes with contractions, and
 \item \label{condi3} $\widehat \nabla_\GVa = \nabla_\GVa$ on $\Gamma(E)$ for all $\GVa \in \fX(M)$.
\end{enumerate}

By \ref{condi2} and \ref{condi3} for $\GVa \in \fX(M)$ and $\GSa \in \Gb E\Delta$ we have
\begin{equation}\label{blarg}
(\widehat \nabla_\GVa \GSa)\cdot \SDa = \widehat\nabla_\GVa (\GSa \cdot \SDa) - \GSa \cdot \nabla_\GVa \SDa \qquad (\SDa \in \Gamma(E^*)).
\end{equation}
Because $\GSa \cdot \alpha$ is an element of $\Gb M\Delta$ it suffices to prescribe how $\widehat\nabla_\GVa$ acts on that space (i.e., on scalar generalized functions), where it should be given by a Lie derivative. We have two Lie derivatives available, but because of \ref{condi1} the map $\widehat \nabla \colon \fX(M) \times \Gb M\Delta \to \Gb M\Delta$ has to be $C^\infty(M)$-linear in $\GVa$, which rules out $\Lieh_\GVa$ for this because it commutes with the embedding, hence cannot be  $C^\infty(M)$-linear in $\GVa$ (see \cite{zbMATH05722389}). Hence, we have to set $\widehat\nabla_\GVa \GFa \coleq \Liet_\GVa \GFa$ for $\GFa \in \Gb M\Delta$ and accordingly define the extension of $\nabla$ to generalized sections of $E$ as
\[ (\widehat \nabla_\GVa \GSa) (\Phi) \coleq \nabla_\GVa ( \GSa(\Phi) ) \qquad (\GVa \in \fX(M), \GSa \in \Gb E\Delta, \Phi \in \VSO \Delta). \]
Because $\GVa \mapsto \widehat\nabla_\GVa \GSa$ for fixed $\GSa \in \Gb E\Delta$ actually is a mapping in
\[ \Hom_{C^\infty(M)} ( \fX(M), \Gb E\Delta) \cong \Hom_{\Gb M\Delta} ( \Gb {TM}\Delta, \Gb E\Delta), \]
the natural extension of $\widehat \nabla$ to generalized directions $\GVa \in \Gb {TM}\Delta$ is given by $(\widehat \nabla_\GVa \GSa)(\Phi) \coleq \nabla_{\GVa(\Phi)} (\GSa(\Phi))$.

In all of this it is crucial that our basic space $\Gb E\Delta$ is big enough to accomodate $\Liet$-derivatives, contrary to the case of \cite{global} or \cite{global2} where there was not even a way to define this Lie derivative, let alone a meaningful covariant derivative. Furthermore, it will be important to see (Proposition \ref{prop_assoc}) that $\Liet_X$ is equal to $\Lieh_X$ for embedded distributions on the level of association.

Motivated by these preliminary considerations we now give the general definition of covariant derivatives in our setting.

\begin{definition}\label{def_covarder}A \emph{generalized covariant derivative} on a vector bundle $E$ is a mapping $\nabla \colon \Gb {TM}\Delta \times \Gb E\Delta \to \Gb E\Delta$ such that for $\GVa, \GVb \in \Gb {TM}\Delta$, $\GSa,\GSb \in \Gb E\Delta$ and $\GFa \in \Gb M \Delta$,
 \begin{enumerate}[label=(\roman*)]
  \item $\nabla_{\GVa+\GVb}\GSa = \nabla_\GVa \GSa + \nabla_\GVb \GSa$,
  \item $\nabla_{\GFa \GVa} \GSa = \GFa \nabla_\GVa \GSa$,
  \item $\nabla_\GVa (\GSa+\GSb) = \nabla_\GVa \GSa + \nabla_\GVa \GSb$,
  \item $\nabla_\GVa (\GFa\GSa) = (\Liet_\GVa \GFa) \GSa + \GFa \nabla_\GVa \GSa$.
 \end{enumerate}
\end{definition}

$\nabla$ extends in a unique way to a derivation on $\Gtb E\Delta$ such that $\nabla_\GVa \GSa = \Liet_\GVa \GSa$ for $\GSa \in \Gb M\Delta$.  

Obviously, the curvature is well-defined with the usual formula as follows:

\begin{definition}Let $\nabla$ be a generalized covariant derivative on $M$.
The \emph{curvature tensor} $\Riem \in \Gb{\Hom(\Lambda^2TM \otimes E, E)}\Delta$ of 
$\nabla$ is defined by
\[ \Riem(\GVa,\GVb)\GSb \coleq \nabla_\GVa \nabla_\GVb \GSb - \nabla_\GVb\nabla_\GVa \GSb - \nabla_{[\GVa,\GVb]} \GSb \qquad (\GVa,\GVb \in \Gb {TM}\Delta, \GSb \in \Gb E\Delta. \]
\end{definition}

As discussed above, in order to view a smooth covariant derivative on $M$ as a generalized one we apply it componentwise, i.e., for fixed $\Phi$. This is exactly the approach used in the special algebra (\cite{genpseudo}). From now on we write $\nabla$ instead of $\widehat\nabla$.

\begin{definition}Any smooth covariant derivative $\nabla$ on $E$ extends to a generalized covariant derivative on $E$ by defining, for $\GVa \in \Gb {TM}\Delta$ and $\GSa \in \Gb E\Delta$,
\[ (\nabla_\GVa \GSa)(\Phi) \coleq \nabla_{\GVa(\Phi)} ( \GSa(\Phi) ) \qquad (\Phi \in \VSO \Delta). \]
\end{definition}

Suppose we are given two generalized covariant derivatives $\nabla$ and $\nabla'$ on $E$. Then $\GSb(\GVa,\GSa) \coleq \nabla_\GVa \GSa - \nabla'_\GVa \GSa$ is in fact $\Gb M\Delta$-bilinear, as is easily verified, thus defines an element of $\Gb{T^*M \otimes E^* \otimes E}\Delta$. Conversely, for any element $\GSb$ of that space and a generalized covariant derivative $\nabla$, $\nabla + \GSb$ again is a generalized covariant derivative. In other words, the space of generalized connections on $E$ is an affine space over $\Gb{T^*M \otimes E^* \otimes E}\Delta$.

We will now transfer the notion of generalized covariant derivative to the quotient.

\begin{definition}
 A generalized covariant derivative $\nabla$ is called \emph{moderate} if $\nabla_\GVa \GSa \in \Gm E \Delta$ for all $\GVa \in \fX(M)$ and $\GSa \in \Gm E \Delta$.
\end{definition}

\begin{proposition}\label{asdfasdf}
 \begin{enumerate}[label=(\roman*)]
  \item\label{modcd.1} Every smooth covariant derivative is moderate.
  \item\label{modcd.2} If $\nabla$ is moderate then $\nabla_\GVa \GSa \in \Gm E \Delta$ for all $\GVa \in \Gm{TM}\Delta$ and $\GSa \in \Gm E\Delta$.
  \item\label{modcd.3} $\nabla$ is moderate if and only if $\nabla - \nabla' \in \Gm{T^*M \otimes E^* \otimes E}\Delta$ for one (hence any) moderate generalized covariant derivative $\nabla'$.
  \item\label{modcd.4} If $\nabla$ is moderate, $\GVa \in \Gm{TM}\Delta$ and $\GSa \in \Gm E \Delta$, then $\nabla_\GVa \GSa \in \Gn E\Delta$ if $\GVa$ or $\GSa$ is negligible.
 \end{enumerate}
\end{proposition}
\begin{proof}
\ref{modcd.1}: Let $\GVa \in \fX(M)$ and $\GSa \in \Gm E\Delta$. Then $\nabla_\GVa\GSa$ is moderate if and only if
\begin{equation}\label{modcd.eq}
(\nabla_\GVa \GSa)\cdot \SDa = \Liet_\GVa (\GSa \cdot \SDa) - \GSa \cdot \nabla_\GVa \SDa
\end{equation}
is moderate for all $\SDa \in \Gamma(E^*)$, which obviously is the case.

\ref{modcd.2}: For $\GSa \in \Gm E \Delta$, $[ \GVa \mapsto \nabla_\GVa \GSa ]$ is an element of \[ \Hom_{C^\infty(M)} ( \fX(M), \Gm E \Delta) \cong \Hom_{\Gm M \Delta} ( \Gm {TM}\Delta, \Gm E \Delta). \]

\ref{modcd.3}: This is clear from $\nabla_\GVa \GSa = (\nabla - \nabla')_\GVa \GSa + \nabla'_\GVa\GSa$.

\ref{modcd.4}: For any a smooth covariant derivative $\nabla'$, $\nabla_\GVa\GSa = \nabla'_\GVa \GSa + S(\GVa,\GSa)$ for some $S \in \Gm{T^*M \otimes E^* \otimes E}\Delta$. From this and \eqref{modcd.eq} the claim follows.
\end{proof}

Hence, the following is well-defined.

\begin{definition}Given a moderate generalized covariant derivative $\nabla$ on $E$, its action on $\Gq E \Delta$ is defined as $\nabla_{[X]} [R] \coleq [\nabla_X R]$, where $[X] \in \Gq{TM}\Delta$ and $R \in \Gq E\Delta$.
\end{definition}

We will also consider the locality condition of Definition \ref{def_locprop} for covariant derivatives.

\begin{definition}A generalized covariant derivative $\nabla$ on $E$ is said to be \emph{local} if $\nabla_X R \in \Gbloc E \Delta$ for all $X \in \fX(M)$ and $R \in \Gbloc E \Delta$.
\end{definition}

\begin{proposition}
 \begin{enumerate}[label=(\roman*)]
  \item\label{loccd.1} Every smooth covariant derivative is local.
  \item\label{loccd.2} If $\nabla$ is local then $\nabla_X R \in \Gbloc E \Delta$ for all $X \in \Gbloc {TM}\Delta$ and $R \in \Gbloc E \Delta$.
  \item\label{loccd.3} $\nabla$ is local if and only if $\nabla - \nabla' \in \Gbloc {T^*M \otimes E^* \otimes E}\Delta$ for one (hence any) local generalized covariant derivative.
 \end{enumerate}
\end{proposition}
\begin{proof}
 \ref{loccd.1}: If $|_U \circ \Phi = |_U \circ \Phi'$, then $(\nabla_X R)(\Phi)|_U = (\nabla_X R(\Phi))|_U = \nabla_{X|_U} R(\Phi)|_U = \nabla_{X|_U} R(\Phi')|_U = (\nabla_X R)(\Phi')|_U$.

\ref{loccd.2} and \ref{loccd.3} are seen as in Proposition \ref{asdfasdf}.
\end{proof}

By the above, a generalized covariant derivative which is local and moderate is well-defined as a map $\Gloc{TM}\Delta \times \Gloc E\Delta \to \Gloc E\Delta$. In particular, this applies to smooth covariant derivatives.

If $\nabla$ is a local generalized covariant derivative on $E$ then for $R \in \Gbloc{U, E}\Delta$ and $X \in \Gbloc{U, TM}\Delta$, $(\nabla_X R)|_V = \nabla_{X|_V} R|_V$.

For a local moderate generalized covariant derivative $\nabla$ on $E$,
we recover the well-known coordinate formula on an open subset $U \subseteq M$,
 \[ (\nabla_X R)|_U = ( X^i \cdot \frac{\pd R^k}{\pd_i} + X^i R^j \Gamma_{ij}^k ) b_k \]
where $(b_i)_i$ is a basis of $\Gamma(U,E)$, $(\beta^j)_j$ its dual basis, and $\pd R^k/\pd_i \coleq \Liet_{\pd_i} R^k$ as well as $\Gamma_{ij}^k \coleq (\nabla_{\pd_i} b_j) \cdot \beta^k$.

As a very basic outlook to the development of pseudo-Riemannian geometry in this setting, we will now describe how a generalized Levi-Civita derivative can be obtained. For this purpose, start with a metric tensor $g$ which can be an element either of $\Gbloc{T^0_2M}\Delta$, $\Gb{T^0_2M}\Delta$, $\Gloc{T^0_2M}\Delta$ or $\Gq{T^0_2M}\Delta$, depending on the concrete application. In any case, we call $g$ \emph{non-singular} if the induced mapping $\widetilde g$ defined by $\widetilde g(X)(Y) \coleq g(X,Y)$ on the corresponding spaces is bijective. Following the classical proof (see e.g.\cite[\S 3]{oneill} and also \cite[Theorem 3.2.82]{GKOS}) one can then show the existence of a unique generalized covariant derivative $\nabla\colon (X,Y) \mapsto (X,Y)$ where $X,Y$ are elements of $\Gbloc{TM}\Delta$, $\Gb{TM}\Delta$, $\Gloc{TM}\Delta$ or $\Gq{TM}\Delta$, accordingly, satisfying the conditions
\[ \nabla_X Y - \nabla_Y X = [X,Y]\text{ and } \nabla_X g = 0 \]
for all $X,Y$.
Moreover, if $g$ is a smooth metric $\nabla$ coincides with the classical Levi-Civita derivative.

Further topics like generalized curves, flows and geodesics as well as the connection to results obtained by purely distributional methods will be studied and published separately; it shall suffice for the moment to point out the significance of our approach in this context.

The largest reasonable class of metrics one can work with in a distributional setting, in particular such that the curvature tensor is well-defined as a distribution, is the class of gt-regular metrics introduced by Geroch and Traschen \cite{gerochtraschen}. However, these can only be used to describe solutions of Einstein's equations with singular support having codimension one, which excludes many interesting phenomena. In order to calculate with singularities of higher codimension, one necessarily has to leave distribution theory and work with a theory of nonlinear generalized functions.

Using the special algebra, the curvature of several important singular metrics has been calculated and given a distributional interpretation in terms of assication (see \cite{genrel, JVdistgen}). Manifold-valued generalized functions were used to study geodesics of certain classes of singular space-times (\cite{geodesics, SS:12, SS:13}); another recent development is a solution theory for the Cauchy problem on non-smooth manifolds with weakly singular Lorentzian metrics (\cite{HKS:12}).

These results show that already in the special algebra, the use of nonlinear generalized functions in general relativity can lead to genuinely new results which were not possible using in distribution theory. From this vantage point it is to be expected that the nonlinear theory of generalized sections presented in this article will serve as a basis for further applications in this field. In this setting, it will be possible for the first time to obtain purely geometrical results, not depending on any choice of coordinate system used for the embedding of distributions, and to have both the sheaf property and a meaningful covariant derivative available.

\textbf{Acknowledgments.} The author thanks M.~Kunzinger and J.~Vickers for helpful discussions. This work was supported by the Austrian Science Fund (FWF) projects P23714 and P25064.

\bibliographystyle{abbrv}
\bibliography{master.bib}

Faculty of Mathematics, University of Vienna, Oskar-Morgenstern-Platz 1, 1090 Vienna, Austria\\
E-Mail address: eduard.nigsch@univie.ac.at
\end{document}